\documentclass[a4paper,12pt]{article}
\usepackage[utf8]{inputenc}
\usepackage[english]{babel}
\usepackage{amsmath, amsfonts, amsthm, amssymb}
\usepackage[labelfont=bf]{caption,subcaption}
\usepackage{color,soul}
\usepackage{enumerate}
\usepackage{float}
\usepackage[margin=2.5cm]{geometry}
\usepackage{graphicx}
\usepackage{hyperref}
\usepackage{mathrsfs}
\usepackage{pdfpages}
\usepackage{tikz}
\usepackage{latexsym}
\usepackage[]{algorithm2e}

\newtheorem{theorem}{Theorem}[section]
\newtheorem{lemma}[theorem]{Lemma}
\newtheorem{corollary}[theorem]{Corollary}
\newtheorem{proposition}[theorem]{Proposition}
\theoremstyle{definition}
\newtheorem{definition}[theorem]{Definition}
\newtheorem{example}[theorem]{Example}
\newtheorem{remark}[theorem]{Remark}
\definecolor{red}{RGB}{255,0,0}
\definecolor{green}{RGB}{0,255,0}
\definecolor{blue}{RGB}{0,0,255}
\definecolor{yellow}{RGB}{255,255,0}
\definecolor{cyan}{RGB}{0,255,255}
\definecolor{magenta}{RGB}{255,0,255}
\definecolor{orange}{RGB}{255,128,0}
\DeclareMathOperator{\lcm}{lcm}

\begin{document}


\title{{Powers of monomial ideals and the Ratliff--Rush operation}}

\author{Oleksandra Gasanova}

\maketitle


\begin{abstract}
Powers of (monomial) ideals is a subject that still calls attraction in various ways. In this paper we present a nice presentation of high powers of ideals in a certain class in
$ \mathbb K[x_1, \ldots, x_n]$ and $\mathbb K[[x_1, \ldots, x_n]]$. As an interesting application it leads to an algorithm for computation of the Ratliff--Rush operation on ideals in that class. The Ratliff--Rush operation itself has several applications, for instance, if $I$ is a regular $\mathfrak m$-primary ideal in a local ring $(R,m)$, then the Ratliff--Rush associated ideal $\tilde I$ is the unique largest ideal containing $I$ and having the same Hilbert polynomial as $I$.
\end{abstract}




\section{Introduction}
Let $R$ be a commutative Noetherian ring and $I$ a regular ideal in it, that is, an ideal containing a non-zerodivisor. The Ratliff--Rush ideal associated to $I$ is defined as $\tilde I=\cup_{k\ge 0}(I^{k+1}:I^k)$. For simplicity we will call it the Ratliff--Rush closure of $I$, even though it does not preserve inclusion, as shown in \cite{Sw}. In \cite{R-R} it is proved that $\tilde I$ is the unique largest ideal that satisfies $I^l=\tilde I^l$ for all large $l$. An ideal $I$ is called Ratliff--Rush if $I=\tilde I$. Properties of the Ratliff--Rush closure and its interaction with other algebraic operations have been studied by several authors, see \cite{Jo}, \cite{He}, \cite{R-R}, \cite{Sw}. In particular, we would like to mention the following two results. If $I$ is an $\mathfrak m$-primary ideal in a local ring $(R,\mathfrak m)$, then $\tilde I$ is the unique largest ideal containing $I$ with the same Hilbert polynomial (the length of ($R/I^l$) for sufficiently large $l$) as $I$. It is also known that the associated graded ring $\oplus_{k\ge 0}I^k/I^{k+1}$ has positive depth if and only if all powers of $I$ are Ratliff--Rush (see \cite{He} for a proof). Several unexpected connections of the Ratliff--Rush closure are discussed in \cite{K}, \cite{Bo} and most recently \cite{Ca} and \cite{Tr}. In general, the Ratliff--Rush closure is hard to compute. In \cite {EL} the author presents an algorithm for Cohen-Macaulay Noetherian local rings, which, however, relies on finding generic elements. In this article we describe a constructive algorithm for computing the Ratliff--Rush closure of $\mathfrak m$-primary monomial ideals of a certain class (we will call it a class of good ideals) in $\mathbb K[x_1,\ldots,x_n]$, which also works in the local ring $\mathbb K[[x_1,\ldots,x_n]]$. This is a generalization of algorithms described in \cite{AA} and \cite{Ve}. 

In Section~3 we introduce the notion of a good ideal. The idea is as follows: any $\mathfrak m$-primary monomial ideal has some $x_1^{d_1},\ldots,x_n^{d_n}$ as minimal generators and therefore defines a (non-disjoint) covering of $\mathbb N^n$ with rectangular "boxes" $B_{a_1,\ldots,a_n}$ of sizes $d_1,\ldots,d_n$, where $a_1,\ldots,a_n$ are nonnegative integers and

$$B_{a_1,\ldots,a_n}:=([a_1d_1,(a_1+1)d_1]\times\ldots\times[a_nd_n,(a_n+1)d_n])\cap\mathbb N^n.$$
Then $I$ is called a good ideal if it satisfies the so-called box decomposition principle, namely, if for any positive integer $l$ any minimal generator of $I^l$ belongs to some box $B_{a_1,\ldots,a_n}$ with $a_1+\ldots+a_n=l-1$. We will also discuss a necessary and a sufficient condition for being a good ideal. From this point, unless specifically mentioned, we will work with good ideals.

In Section~4 we will associate an ideal to each box in the following way: if $I$ is a good ideal and $B_{a_1,\ldots,a_n}$ is some box, then it contains some of the minimal generators of $I^l$, where $l=a_1+\ldots+a_n+1$. Since they are in $B_{a_1,\ldots,a_n}$, they are divisible by $(x_1^{d_1})^{a_1}\cdots(x_n^{d_n})^{a_n}$. Therefore, we can define 

$$I_{a_1,\ldots,a_n}:=\left\langle{\frac{m}{(x_1^{d_1})^{a_1}\cdots(x_n^{d_n})^{a_n}}\mid m \in B_{a_1,\ldots, a_n} \cap G(I^l) }\right\rangle.$$ 
We will conclude this section by showing that $$I_{a_1,\ldots, a_n}=I^l:\langle(x_1^{d_1})^{a_1}\cdots(x_n^{d_n})^{a_n}\rangle,$$ which immediately implies the following property: if $(a_1,\ldots,a_n)\le(b_1,\ldots,b_n)$, then $I_{a_1,\ldots, a_n}\subseteq I_{b_1,\ldots, b_n}$.

In Section~5 we will study the asymptotic behaviour of $I_{a_1,\ldots,a_n}$. Now that we know that $I_{a_1,\ldots,a_n}$ grows when $(a_1,\ldots,a_n)$ grows, and given that ideals can not grow forever, we are expecting some sort of stabilization in $I_{a_1,\ldots,a_n}$ when $(a_1,\ldots,a_n)$ is large enough. In other words, we are expecting some pattern on $I^l$ for large $l$.
In Section~6 we will prove the main theorem of this paper, namely, the following: if $I$ is a good ideal, then 
$$\tilde{I}=I_{q_1,0,\ldots,0}\cap I_{0,q_2,\ldots,0}\cap\ldots\cap I_{0,\ldots,0,q_n},$$
where $I_{q_1,0,\ldots,0}$ is the stabilizing ideal of the chain $I_{0,0,\ldots,0}\subseteq I_{1,0,\ldots,0}\subseteq I_{2,0,\ldots,0}\subseteq\ldots$, and $I_{0,q_2,\ldots,0}$ is the stabilizing ideal of the chain $I_{0,0,\ldots,0}\subseteq I_{0,1,\ldots,0}\subseteq I_{0,2,\ldots,0}\subseteq\ldots$, and so on. The pattern eatablished in Section~5 will play an important role in the proof of the main theorem.
In Section~7 we will show that computation of $I_{0,0,\ldots,q_i,0,\ldots,0}$ is much easier than it seems. In particular, we will show that the corresponding chain stabilizes immediately as soon as we have two equal ideals.
Section~8 contains examples and explicit computations of $\tilde I$. We use Singular (\cite{DGPS}) for all our computations.

In Section~9 we discuss how to detect whether a given ideal is a good one if it satisfies the necessary condition and does not satisfy the sufficient condition from Section~3.
In Section~10 we discuss the following question: are powers of good ideals also good? Unfortunately, the answer is negative in most cases. We also give a definition of a very good ideal: if $I_{a_1,\ldots,a_n}=I$ for all $(a_1,\ldots,a_n)$, then such an ideal is called a very good ideal and all its powers are Ratliff--Rush.

In Section~11 we discuss the connection of the above results to Freiman ideals that have been studied in \cite{FR}. In particular, we will show that for an $\mathfrak m$-primary equigenerated monomial ideal being Freiman is equivalent to being very good.

\section{Preliminaries and notation}
Throughout this paper we will work with $R=\mathbb K[x_1,\ldots,x_n]$, although all the results will also hold in the local ring $\mathbb K[[x_1,\ldots,x_n]]$. We will be dealing with monomial ideals $I$ in $R$. We start by listing a few basic properties of monomial ideals that will be used later.
\begin{enumerate}
\item
For each monomial ideal there is a unique minimal generating set consisting of monomials. For an ideal $I$ we denote $G(I)$ to be its minimal monomial generating set.
\item
If $m\in I=\langle m_1,\ldots,m_k\rangle$, where $m$ and all $m_i$ are monomials, then there is some $i$ such that $m_i$ divides $m$.
\item
If $I=\langle m_1,\ldots,m_k \rangle$, $J=\langle n_1,\ldots,n_l\rangle$, then 
$IJ=\langle m_1n_1,\ldots,m_1n_l, \ldots,\\ m_kn_1, \ldots, m_kn_l\rangle$, but this generating set is not minimal in general.
\item
There is a natural bijection between monomials in $\mathbb K[x_1,\ldots,x_n]$ and points in $\mathbb N^n$ in the following way: $x_1^{\alpha_1}x_2^{\alpha_2}\cdots x_n^{\alpha_n} \leftrightarrow (\alpha_1,\alpha_2,\ldots, \alpha_n)$. We will say that $(\beta_1,\beta_2,\ldots, \beta_n) \le (\alpha_1,\alpha_2,\ldots, \alpha_n)$ if $\beta_i\le \alpha_i$ for all $i\in \{1,2,\ldots, n\}$. Then it is clear that $x_1^{\beta_1}x_2^{\beta_2}\cdots x_n^{\beta_n}$ divides $x_1^{\alpha_1}x_2^{\alpha_2}\cdots x_n^{\alpha_n}$ if and only if $(\beta_1,\beta_2,\ldots, \beta_n) \le (\alpha_1,\alpha_2,\ldots, \alpha_n)$ and that multiplication of monomials corresponds to addition of points. We will often say that some monomial belongs to some subset of $\mathbb N^n$, meaning that the corresponding point belongs to that subset. Sometimes we will also say that some point belongs to some ideal $I$, meaning that the corresponding monomial belongs to $I$.
\item
$\langle m_1\rangle:\langle m_2\rangle=\left< \frac{m_1}{\gcd(m_1,m_2)}\right>$.
\item
$I:(J_1+J_2)=(I:J_1)\cap (I:J_2)$ and $(I_1+I_2):\langle m\rangle=I_1:\langle m\rangle+I_2:\langle m\rangle$.
\end{enumerate}
Let $I$ be an  $\mathfrak m$-primary monomial ideal of $R$, where $\mathfrak m=\langle x_1,x_2,\ldots,x_n\rangle$, that is, for some positive integers $d_1,\ldots,d_n$ we have $\{x_1^{d_1},\ldots,x_n^{d_n}\}\subseteq G(I)$. Henceforth, by $I$ we always mean an $\mathfrak m$-primary monomial ideal and denote $\mu_i:=x_i^{d_i}$, $1\le i\le n$. Also, in this paper we do not consider any polynomials other than monomials since it will always be sufficient to prove statements for monomials only.
\section{Good and bad ideals}
In this section we will introduce the notion of a good ideal, prove a necessary and a sufficient condition for being a good ideal and give some examples.
\begin{definition}
Let $I$ be an ideal. Recall that $\{\mu_1,\ldots,\mu_n\}\subseteq G(I)$, where $\mu_i=x_i^{d_i}$ for some $d_i$. Let $a_1,\ldots, a_n$ be nonnegative integers and denote
$$B_{a_1,\ldots,a_n}:=([a_1d_1,(a_1+1)d_1]\times\ldots\times[a_nd_n,(a_n+1)d_n])\cap \mathbb N^n.$$
$B_{a_1,\ldots,a_n}$ will be called the {\bf {box}} with coordinates $(a_1,\ldots,a_n)$, associated to $I$. Points of the type $(k_1d_1,\ldots,k_nd_n)$ and the corresponding monomials, where all $k_i$ are nonnegative integers, will be called {\bf {corners}}. We will mostly work with one ideal at a time, thus there is no need to use any additional index to show that $B_{a_1,\ldots,a_n}$ depends on $I$. Note that all minimal generators of $I$ lie in $B_{0,\ldots,0}$.
\end{definition}

\begin{definition}
We will say that an ideal $I$ satisfies the {\bf{box decomposition principle}} if the following holds: for every positive integer $l$, every minimal generator of $I^l$ belongs to some box $B_{a_1,\ldots,a_n}$ such that $a_1+\ldots+a_n=l-1.$ Ideals satisfying the box decomposition principle will be called {\bf {good}}, otherwise they will be called {\bf {bad}}.
\end{definition}
\begin{example}
Consider the ideal $I=\langle x^3,y^3,z^3, xyz\rangle$  in $\mathbb K[x,y,z]$. Then $x^2y^2z^2$ is a minimal generator of $I^2$, but it only belongs to $B_{0,0,0}$ and $0+0+0 \not= 1$. Therefore, $I$ is a bad ideal.
\end{example}
\begin{example}
Let $I=\langle x^3,y^3,z^3, x^2y^2z^2\rangle$  in $\mathbb K[x,y,z]$. Then $$G(I^2)=\{x^6,y^6,z^6,x^3y^3,x^3z^3,y^3z^3,x^5y^2z^2, x^2y^5z^2, x^2y^2z^5\}.$$ Note that the square of $x^2y^2z^2$ is not a minimal generator, thus we are not examining it. Below we list all the possible boxes with sums of coordinates equal to $1$ and minimal generators of $I^2$ that belong to these boxes:
$$B_{1,0,0}: x^6,x^3y^3,x^3z^3,x^5y^2z^2,$$
$$B_{0,1,0}: y^6,x^3y^3,y^3z^3,x^2y^5z^2,$$
$$B_{0,0,1}: z^6,x^3z^3,y^3z^3,x^2y^2z^5.$$
Note that each minimal generator of $I^2$ belongs to at least one such box. For simplicity, we denote $$S_{1,0,0}:=\{x^6,x^3y^3,x^3z^3,x^5y^2z^2\}$$ (minimal generators of $I^2$ that belong to $B_{1,0,0}$) and we similarly define $S_{0,1,0}$ and $S_{0,0,1}$. We see that elements in $S_{1,0,0}$ are multiples by $\mu_1=x^3$ of the minimal generators of $I$ (similarly for $S_{0,1,0}$ and $S_{0,0,1}$), that is, $$I^2=\langle S_{1,0,0}, S_{0,1,0}, S_{0,0,1}\rangle=\mu_1I+\mu_2I+\mu_3I.$$ Geometrically it means that $I^2$ is minimally generated by all appropriate translations of $I$.

What happens in $I^3$ and higher powers? It is easy to see that the situation is quite similar there as well. Say, for $I^3$ we take products of minimal generators of $I$ with minimal generators of $I^2$ (which are translations of the minimal generators of $I$). Obviously, we will get nothing but larger translations of $I$, that is, $I^3=\mu_1^2I+\mu_2^2I+\mu_3^2I+\mu_1\mu_2I+\mu_1\mu_3I+\mu_2\mu_3I$. The first summand corresponds to the minimal generators in $B_{2,0,0}$, the second one -- to those in $B_{0,2,0}$, the third one -- to those in $B_{0,0,2}$, the fourth one -- to those in $B_{1,1,0}$, the fifth one -- to those in $B_{1,0,1}$, the sixth one -- to those in $B_{0,1,1}$. Clearly, the pattern repeats in all powers of $I$: for every $l\ge 1$ we have
$$I^l=\sum_{l_1+\ldots+l_n=l-1}\mu_1^{l_1}\ldots \mu_n^{l_n}I.$$
Therefore, $I$ is a good ideal.
\end{example}
\begin{proposition}
\label{remeq}
The following are equivalent:
\begin{enumerate}[(1)]
\item I is a good ideal;
\item for any $l\ge 1$ and for any $m\in I^l$ there exist $a_1,\ldots, a_n$ such that $m\in B_{a_1,\ldots, a_n}$ and $a_1+\ldots+a_n\ge l-1$;
\item for any $l\ge 1$ and for any $m_1,\ldots, m_l\in G(I)$ there exist $a_1,\ldots, a_n$ such that \\ $m_1\cdots m_l\in B_{a_1,\ldots, a_n}$ and $a_1+\ldots+a_n\ge l-1$.
\end{enumerate}
\end{proposition}
\begin{proof}
\hfill
\begin{description}

\item [(1)$\Rightarrow $(2):]
Let $I$ be a good ideal and let $l\ge 1$ and $m\in I^l$. Then $m$ is divisible by some $m_1\in G(I^l)$ and $m_1\in B_{b_1,\ldots,b_n}$ for some $b_1,\ldots,b_n$ with $b_1+\ldots+b_n=l-1$. Then there exist $a_1,\dots,a_n$ such that $(a_1,\ldots,a_n)\ge(b_1,\ldots, b_n)$ (thus $a_1+\ldots+a_n\ge l-1$) and $m\in B_{a_1,\ldots, a_n}$. 
\item[(2)$\Rightarrow $(3):] Obvious.
\item[(3)$\Rightarrow $(1):] Let $l\ge 1$ and $m\in G(I^l)$. We want to show that there is a box $B_{b_1,\ldots,b_n}$ such that $m\in B_{b_1,\ldots,b_n}$ and $b_1+\ldots+b_n=l-1$. Since $m\in G(I^l)$, we have $m=m_1\cdots m_l$ for some $m_1,\ldots, m_l\in G(I)$. Then there exist $a_1,\ldots, a_n$ such that $m=m_1\cdots m_l\in B_{a_1,\ldots, a_n}$ and $a_1+\ldots+a_n\ge l-1$. If we assume that $a_1+\ldots+a_n\ge l$, then $m$ is divisible by $\mu_1^{a_1}\cdots\mu_n^{a_n}\in I^l$. We have two cases:
\begin{enumerate}
\item If $m\not=\mu_1^{a_1}\cdots\mu_n^{a_n}$, it contradicts $m\in G(I^l)$ and thus $a_1+\ldots+a_n=l-1$ and we can set $b_i:=a_i$ for all $i$. 
\item If $m=\mu_1^{a_1}\cdots\mu_n^{a_n}$, then $a_1+\ldots+a_n=l$ since $m$ can not possibly belong to $G(I^l)$  if $a_1+\ldots+a_n>l$. Note that we are not proving that $\mu_1^{a_1}\cdots\mu_n^{a_n}\in G(I^l)$ if $a_1+\ldots+a_n=l$ (this will be done later in this paper). In any case, we will find an appropriate box for $m$. At least one of $a_i$ is different from $0$. Without loss of generality, we can assume $a_1\ge 1$. Then $m\in B_{a_1-1,a_2,\ldots, a_n}$ and $(a_1-1)+a_2+\ldots+a_n=l-1$.
\end{enumerate}
\end{description}

\end{proof}

Now we are interested in some necessary and sufficient conditions on $G(I)$ for an ideal $I$ to be good. Note that every monomial ideal in $\mathbb K[x]$ is a good one.
\begin{theorem}{(A necessary condition for being a good ideal)} Let $I$ be an ideal in $\mathbb K[x_1,\ldots,x_n]$. If $I$ is a good ideal, then for any minimal generator $x_1^{\alpha_1}x_2^{\alpha_2}\cdots x_n^{\alpha_n}$ of $I$ the following holds:  $$\frac{\alpha_1}{d_1}+\cdots+\frac{\alpha_n}{d_n}\ge1.$$
\end{theorem}
\begin{proof} Assume that there is a minimal generator for which the above condition fails, that is, $m=x_1^{\alpha_1}x_2^{\alpha_2}\cdots x_n^{\alpha_n}$ with $\frac{\alpha_1}{d_1}+\cdots+\frac{\alpha_n}{d_n}=1-\epsilon$ for some $0<\epsilon<1$. Let $l$ be a positive integer such that $l>\frac{1}{\epsilon}$. We will show that the box decomposition principle fails for $I^l$.
Consider $m^l=x_1^{l\alpha_1}x_2^{l\alpha_2}\cdots x_n^{l\alpha_n}$. The first coordinate of $m^l$ is $l\alpha_1$, therefore, the first coordinate of the box where $m^l$ belongs is at most $\lfloor\frac{l\alpha_1}{d_1}\rfloor$ and similar inequalities hold for the other coordinates. Therefore, the sum of coordinates of any box containing $m^l$ is less or equal to $\lfloor\frac{l\alpha_1}{d_1}\rfloor+\ldots+\lfloor\frac{l\alpha_n}{d_n}\rfloor\le\frac{l\alpha_1}{d_1}+\ldots+\frac{l\alpha_n}{d_n}=l(\frac{\alpha_1}{d_1}+\cdots+\frac{\alpha_n}{d_n})=l(1-\epsilon)<l(1-\frac{1}{l})=l-1.$
Thus all boxes containing $m^l$ have the sum of coordinates strictly less than $l-1$ and we are done by Proposition~\ref{remeq}.
\end{proof}

\begin{theorem}{(A sufficient condition for being a good ideal)}

Let $I$ be an ideal in $\mathbb K[x_1,\ldots,x_n]$. Assume that for any minimal generator \\ $x_1^{\alpha_1}x_2^{\alpha_2}\cdots x_n^{\alpha_n}$ of $I$ which is not a corner the following holds:
$$\frac{\alpha_1}{d_1}+\cdots+\frac{\alpha_n}{d_n}\ge \frac{n}{2}.$$
Then $I$ is a good ideal.
\end{theorem}
\begin{proof}
The claim is trivial for $n=1$, thus assume $n\ge 2$.
Let $m_1,m_2\in G(I)$, where $m_1=x_1^{\alpha_1}\cdots x_n^{\alpha_n}$, $m_2=x_1^{\beta_1}\cdots x_n^{\beta_n}$ with $\frac{\alpha_1}{d_1}+\cdots+\frac{\alpha_n}{d_n}\ge \frac{n}{2}$ 
and $\frac{\beta_1}{d_1}+\cdots+\frac{\beta_n}{d_n}\ge \frac{n}{2}$. 
By Proposition~\ref{remeq}, it suffices to show that $m_1m_2=\mu_ix_1^{\gamma_1}\cdots x_n^{\gamma_n}$ for some $i$ and with $\frac{\gamma_1}{d_1}+\cdots+\frac{\gamma_n}{d_n}\ge \frac{n}{2}$. Note that $\frac{\alpha_1+\beta_1}{d_1}+\cdots+\frac{\alpha_n+\beta_n}{d_n}\ge n$, thus we must have $\frac{\alpha_i+\beta_i}{d_i}\ge 1$ for some $i$. We can assume $i=1$, then $\frac{\alpha_1+\beta_1-d_1}{d_1}+\cdots+\frac{\alpha_n+\beta_n}{d_n}\ge n-1\ge \frac{n}{2}$. Setting $\gamma_1=\alpha_1+\beta_1-d_1$ and $\gamma_i=\alpha_i+\beta_i$ for $2\le i \le n$ finishes the proof.

\end{proof}
\begin{remark}
For $n=2$ the necessary condition is equivalent to the sufficient condition.
\end{remark}
\begin{example}(A good ideal that does not satisfy the sufficient condition)

Let $I=\langle \mu_1,\mu_2,\mu_3,m\rangle=\langle x^5,y^5,z^5,xyz^4\rangle\subset \mathbb K[x,y,z]$. The ideal satisfies the necessary condition, but not the sufficient condition, so we will explore it by hand. What kinds of generators do we have in $I^l$? First of all, we notice that $m^5=x^5y^5z^{20}$ is divisible by, say, $\mu_1\mu_2\mu_3^3\in I^5$, therefore, it is not a minimal generator of $I^5$. Therefore, for any $l$, the minimal generators of $I^l$ will be of the form $\mu_1^{k_1}\mu_2^{k_2}\mu_3^{k_3}m^k$, where $k_1+k_2+k_3+k=l$ and $k\le 4$. If $k=0$, the monomial is just a corner and this case is trivial, so let $k\ge 1$. Clearly, such a monomial belongs to a box whose sum of coordinates is $l-1$ if and only if $m^k$ belongs to a box whose sum of coordinates is $k-1$. So the only thing we need to check is whether $m^k$ belongs to a box whose sum of coordinates is $k-1$, $2\le k \le 4$ (this is always true for $k=1$). We see that $m^2=x^2y^2z^8\in B_{0,0,1}$, $m^3=x^3y^3z^{12}\in B_{0,0,2}$, $m^4=x^4y^4z^{16}\in B_{0,0,3}$. Therefore, $I$ is a good ideal.
\end{example}
\begin{example}(A bad ideal that satisfies the necessary condition)

Let $I=\langle x^5,y^5,z^5,x^2y^2z^2\rangle\subset \mathbb K[x,y,z]$. The ideal satisfies the necessary condition, but not the sufficient condition. We see that $x^4y^4z^4$ is a minimal generator of $I^2$ and it only belongs to $B_{0,0,0}$. Since $0+0+0 \not= 1$, $I$ is a bad ideal.
\end{example}
Ideals that satisfy the necessary condition, but do not necessarily satisfy the sufficient condition, will be discussed further in Section~9 of this paper. There is a way to detect whether an ideal is good or bad and it basically uses the ideas from the two examples above.
\section{Ideals inside boxes and their connection to each other}
We will start this section with an example aimed to give a motivation for the future constructions.
\begin{example}
\label{ex}
Let $I=\langle x^5,y^5,xy^4,x^4y\rangle\subset \mathbb K[x,y]$. $I$ is a good ideal since it satisfies the sufficient condition. In this case the associated boxes have sizes $5\times 5$. Figure~1 represents powers of $I$ up to $I^4$.
Consider the box $B_{1,0}$. Inside this box we see some of the minimal generators of $I^2$, namely, $\{x^5y^5,x^6y^4,x^8y^2,x^9y,x^{10}\}$.
Since they are in $B_{1,0}$, they are divisible by $\mu_1^1\mu_2^0=x^5$. 
If we divide all these monomials by $\mu_1^1\mu_2^0$, we will get $\{y^5,xy^4,x^3y^2,x^4y,x^5\}$.

Define $I_{1,0}:=\langle y^5,xy^4,x^3y^2,x^4y,x^5\rangle$. Geometrically, this means viewing monomials in $B_{1,0}$ as if the lower left corner of $B_{1,0}$ was the origin. In this particular example we have $I_{0,0}=I$, $I_{1,0}=\langle y^5,xy^4,x^3y^2,x^4y,x^5\rangle$, $I_{0,1}=\langle y^5,xy^4,x^2y^3,x^4y,x^5\rangle$, $I_{a,b}=\langle y^5,xy^4,x^2y^3,x^3y^2,x^4y,x^5\rangle$ for all other $(a,b)$. 

\begin{figure}[H]
\centering
\begin{tikzpicture}[scale=0.4]
\draw[step=1cm,gray,opacity=0.3,line width=0.001 mm] (0,0) grid (20,20);

\draw[step=5cm,opacity=0.3] (0,0) grid (20,20);

\draw[violet, thick] (0,5)--(1,5)--(1,4)--(4,4)--(4,1)--(5,1)--(5,0);
\draw[violet, thick] (0,10)--(1,10)--(1,9)--(2,9)--(2,8)--(4,8)--(4,6)--(5,6)--(5,5)--(6,5)--(6,4)--(8,4)--(8,2)--(9,2)--(9,1)--(10,1)--(10,0);
\foreach \x in {0,...,14}{
\draw[violet, thick] (\x,15-\x)--(\x+1,15-\x)--(\x+1,14-\x);
}
\foreach \x in {0,...,19}{
\draw[violet, thick] (\x,20-\x)--(\x+1,20-\x)--(\x+1,19-\x);
}
\foreach \x in {0,...,4}{
    \node [below, very thin] at (5*\x,0) {$_\x$};
    \node [left, very thin] at (0,5*\x) {$_\x$};
    }
\end{tikzpicture}

\caption{powers of $I$: $I$, $I^2$, $I^3$ and $I^4$} 
\end{figure}
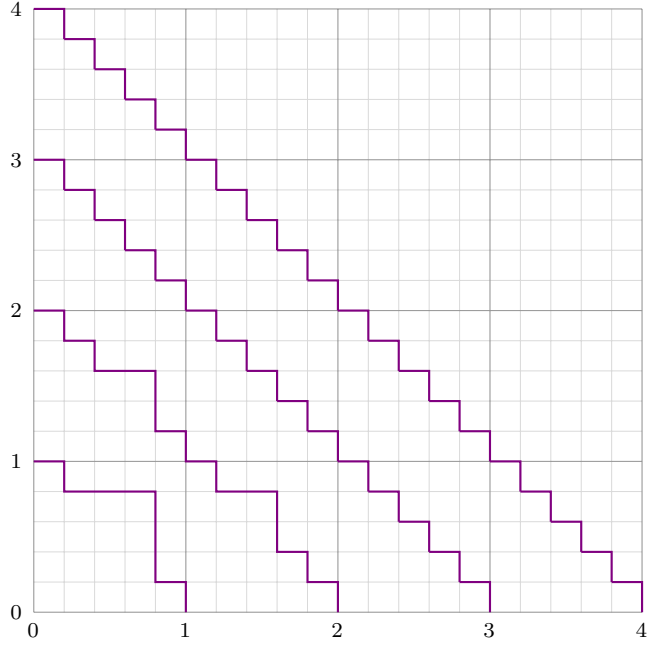

\end{example}

This gives rise to a more general definition.

\begin{definition}
Let $I$ be a good ideal and $a_1,\ldots, a_n$ nonnegative integers. We define $$I_{a_1,\ldots,a_n}:=\left\langle{\frac{m}{\mu_1^{a_1}\cdots\mu_n^{a_n}}\mid m \in B_{a_1,\ldots, a_n} \cap G(I^l) }\right\rangle,$$
where $l=a_1+\ldots+a_n+1$. Note that this a minimal generating set of $I_{a_1,\ldots,a_n}$.
\end{definition}

A priori it is not clear why, given a good ideal $I$, any box $B_{a_1,\ldots,a_n}$ has a nonempty intersection with $G(I^l)$, where $l=a_1+\ldots+a_n+1$. The next remark will in particular show that intersections of this type are never empty.

\begin{remark}{(Corners are needed)}
\label{rem}
Let $I$ be a good ideal, let $m=\mu_1^{k_1}\cdots\mu_n^{k_n}$ be some corner and put $l:=k_1+k_2+\ldots+k_n$. Then $m \in G(I^l).$ Indeed, assume that $m$ is not a minimal generator of $I^l$, which means that there exists a strictly smaller generator $s=x_1^{s_1}\cdots x_n^{s_n}$. Note that $s$ is a product of $l$ minimal generators of $I$, and since $I$ is a good ideal, the necessary condition holds and therefore $\frac{s_1}{d_1}+\ldots+\frac{s_n}{d_n}\ge l$. As for $m$, we have $\frac{k_1d_1}{d_1}+\ldots+\frac{k_nd_n}{d_n}=k_1+\ldots+k_n=l$, which is a contradiction, since $s$ strictly divides $m$.

Now we see that, given a good ideal $I$ and a box $B_{a_1,\ldots, a_n}$, the box necessarily contains, for instance, all monomials of the type 
$\left\{\mu_j\prod_{i=1}^{n}\mu_i^{a_i}\mid 1\le j\le n
\right\}$. All these monomials are corners, therefore, they are minimal generators of $I^l$, where $l=a_1+\ldots+a_n+1$. Thus we conclude that any box $B_{a_1,\ldots, a_n}$ has a nonempty intersection with $G(I^l)$, since this intersection contains $n$ corners, mentioned above. As a consequence, $\mu_1,\ldots,\mu_n$ are minimal generators of any $I_{a_1,\ldots, a_n}$.
\end{remark}
\begin{proposition}
\label{prop1}
Let $I$ be a good ideal and $a_1,\ldots,a_n$ nonnegative integers. Then $$I_{a_1,\ldots, a_n}=I^l:\langle\mu_1^{a_1}\cdots \mu_n^{a_n}\rangle,$$ where $l=a_1+\ldots+a_n+1$.
\end{proposition}
\begin{proof} It is clear from the definition that $I_{a_1,\ldots, a_n}\subseteq I^l:\langle\mu_1^{a_1}\cdots \mu_n^{a_n}\rangle$. For the other inclusion, let $m\in I^l:\langle\mu_1^{a_1}\cdots \mu_n^{a_n}\rangle$. Then $m\mu_1^{a_1}\cdots \mu_n^{a_n}\in I^l$, that is, $m\mu_1^{a_1}\cdots \mu_n^{a_n}$ is a multiple of some $g\in G(I^l)$, say, $m\mu_1^{a_1}\cdots \mu_n^{a_n}=gg_1$. Being a minimal generator of $I^l$, $g$ belongs to some box, say, $B_{b_1,\ldots,b_n}$, with $b_1+\ldots+b_n=l-1=a_1+\ldots+a_n$. If $(a_1,\ldots,a_n)=(b_1,\ldots,b_n)$, then $m$ is a multiple of $\frac{g}{\mu_1^{a_1}\cdots\mu_n^{a_n}}$, which is a generator of $I_{a_1,\ldots, a_n}$ and thus we are done. If $(a_1,\ldots,a_n)\not=(b_1,\ldots,b_n)$, then there is some $a_i<b_i$. Without loss of generality, we assume that $a_1<b_1$. Then the right hand side of $m\mu_1^{a_1}\cdots \mu_n^{a_n}=gg_1$ is divisible by $\mu_1^{b_1}$, thus $m$ is divisible by $\mu_1$, and $\mu_1$ is a minimal generator of $I_{a_1,\ldots, a_n}$ by Remark~\ref{rem}. Therefore, $m\in I_{a_1,\ldots, a_n}$.
\end{proof}
Let $a_1,a_2,\ldots,a_n$ and $b_1,b_2,\ldots, b_n$ be nonnegative integers such that \\ $(a_1,\ldots,a_n)\le(b_1,\ldots,b_n)$. Since $I_{a_1,\ldots, a_n}=I^{a_1+\ldots+a_n+1}:\langle\mu_1^{a_1}\cdots \mu_n^{a_n}\rangle$ and $I_{b_1,\ldots, b_n}=I^{b_1+\ldots+b_n+1}:\langle\mu_1^{b_1}\cdots \mu_n^{b_n}\rangle$, we immediately conclude the following:

\begin{corollary}
\label{cor}
Let $I$ be a good ideal and let $a_1,a_2,\ldots,a_n$ and $b_1,b_2,\ldots, b_n$ be nonnegative integers such that $(a_1,\ldots,a_n)\le(b_1,\ldots,b_n)$. Then $I_{a_1,\ldots, a_n}\subseteq I_{b_1,\ldots, b_n}$.
\end{corollary}

\section{Asymptotic behaviour of $I_{a_1,\ldots, a_n}$}
Now we know that $I_{a_1,\ldots,a_n}$ grows as $(a_1,\ldots,a_n)$ grows. Since  $I_{a_1,\ldots,a_n}$ can not increase forever, one expects some pattern on high powers of $I$, which is indeed the case. Let us take a closer look at the situation.
\begin{definition}
Let $a_1,\ldots,a_n$ be nonnegative integers. We will use the following notation: 
\begin{multline*}
C_{\underline{a_1},\underline{a_2},\ldots,\underline{a_k},a_{k+1},a_{k+2},\ldots,a_n}:=\{(b_1,\ldots,b_n)\in\mathbb N^n \mid \\b_1=a_1,\ldots,b_k=a_k, b_{k+1}\ge a_{k+1},\ldots,b_{n}\ge a_{n}\}.
\end{multline*}
We will use a similar notation for any configuration of fixed and non-fixed coordinates. Sets of this type will be called {\bf {cones}}, for any cone the number of non-fixed coordinates will be called its {\bf dimension} and $(a_1,\ldots,a_n)$ will be called its {\bf vertex}. Note that $\mathbb N^n=C_{0,0,\ldots,0}$.
\end{definition}
\begin{example}
Let $n=3$ and $a_1=1,a_2=2,a_3=3$. Then $C_{1,\underline{2},3}=\{(b_1,2,b_3)\mid b_1\ge 1,b_3\ge3\}$ and the dimension of this cone is $2$. 
\end{example}
\begin{definition}
Let $a_1,\ldots,a_n$ be nonnegative integers. By $A_{a_1,\ldots,a_n}$ we denote the set of all cones that satisfy the following conditions:
\begin{enumerate}

\item 
if $(b_1,\ldots,b_n)$ 
is the vertex of a cone in $A_{a_1,\ldots,a_n}$, then $b_i\le a_i$ for all $1\le i\le n$;
\item for all $1\le i\le n$ the following holds: if $b_i=a_i$, then $b_i$ is not underlined and if $b_i<a_i$, then $b_i$ is underlined.
\end{enumerate}
Note that the unique cone of dimension $n$ in $A_{a_1,\ldots,a_n}$ is $C_{a_1,\ldots,a_n}$.

\end{definition}

\begin{example}
\label{ex2}
Let $n=2$, $a_1=2$, $a_2=1$. We would like to find all the cones in $A_{2,1}$. For any cone in $A_{2,1}$ the first coordinate of its vertex can only be chosen from the set $\{0,1,2\}$; we underline it if we choose 0 or 1 and do not underline it if we choose 2. Independently, the second coordinate can only be chosen from the set $\{0,1\}$ and we underline it if we choose 0 and do not underline it if we choose 1. Therefore, we will get six cones in total:

$A_{2,1}=\{\textcolor{cyan}{C_{\underline{0},\underline{0}}},\textcolor{green}{C_{\underline{0},1}},\textcolor{magenta}{C_{\underline{1},\underline{0}}},\textcolor{blue}{C_{\underline{1},1}},\textcolor{orange}{C_{2,\underline{0}}},\textcolor{red}{C_{2,1}}\}$.

\begin{figure}[H]
\centering
\begin{tikzpicture}[scale=0.7]
	\draw[green,very thick]   (0,1) -- (0,7);
	\draw[blue,very thick]    (1,1) -- (1,7);
	\draw[orange,very thick]  (2,0) -- (7,0);
	\draw[red,very thick]     (2,1) -- (7,1);
	\draw[red,very thick]     (2,1) -- (2,7);
	\fill[fill=magenta] (1,0) circle (0.08 cm);
	\fill[fill=cyan] (0,0) circle (0.08 cm);
	\foreach \x in {2,...,7}{

    \fill[fill=orange] (\x,0) circle (0.08 cm);
    }
    \foreach \x in {1,...,7}{
    \fill[fill=green] (0,\x) circle (0.08 cm);
    }
    \foreach \x in {1,...,7}{
    \fill[fill=blue] (1,\x) circle (0.08 cm);
    }
    \foreach \x in {2,...,7}{
    \foreach \y in {1,...,7}{
    \fill[fill=red] (\x,\y) circle (0.08 cm);
    }}
    \foreach \x in {0,...,7}{
    \node [below, very thin] at (\x,0) {$_\x$};
    \node [left, very thin] at (0,\x) {$_\x$};
    }

\end{tikzpicture}
\caption{cones of $A_{2,1}$}
\end{figure}
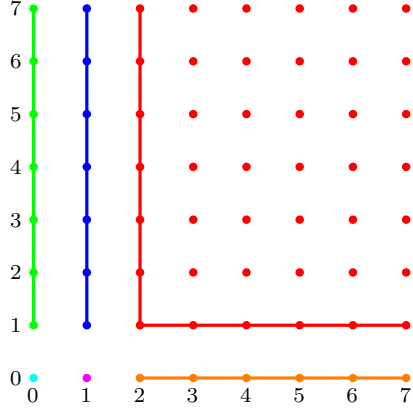

Figure~2 represents the six cones from $A_{2,1}$. The boundary lines are only drawn for better visibility. Clearly, the number of boundary lines equals the dimension of the cone.
\end{example}
\begin{lemma}
Let $a_1,\ldots,a_n$ be nonnegative integers. Then cones in $A_{a_1,\ldots,a_n}$ form a disjoint covering of $\mathbb N^n$.
\end{lemma}
\begin{proof}
Let $\textbf{b}=(b_1,\ldots,b_n)$ be a point in $\mathbb N^n$. We will find a unique cone in $A_{a_1,\ldots,a_n}$ that contains this point. First of all, we compare $a_1$ and $b_1$.
\begin{enumerate}
\item If $b_1\ge a_1$, then the first coordinate of our future cone containing $\textbf{b}$ is non-fixed since otherwise $A_{a_1,\ldots,a_n}$ contains $C_{\underline{b_1},\ldots}$, which is a contradiction with $b_1\ge a_1$. Therefore, our first coordinate has to be non-fixed, hence equal to $a_1$.
\item If $b_1<a_1$, then the first coordinate can only be a fixed one since otherwise it is equal to $a_1$, but $\textbf{b}$ can not belong to $C_{a_1,\ldots}$ since $b_1<a_1$. Therefore, the first coordinate has to be a fixed one, hence equal to $b_1$.
\end{enumerate}
Proceeding in the same way we construct a cone in $A_{a_1,\ldots,a_n}$ that contains $(b_1,\ldots,b_n)$. From the construction it is clear that this cone is unique, which finishes our proof.
\end{proof}

We have seen that given $\mathbb N^n=C_{0,0,\ldots,0}$ and a point $(a_1,\ldots,a_n)\in C_{0,0,\ldots,0}$, we can decompose $C_{0,0,\ldots,0}$ into a disjoint union of cones, associated to this point, where the unique cone of dimension $n$ is $C_{a_1,\ldots,a_n}$ and all other cones have strictly lower dimensions. It is not hard to see that we can replace $\mathbb N^n=C_{0,0,\ldots,0}$ with any other cone and replace $(a_1,\ldots,a_n)$ with any point in this cone and have a similar decomposition. First of all, assume that all coordinates of this cone are non-fixed, say, we have $C_{s_1,\ldots,s_n}$ and a point $(s_1+k_1,\ldots,s_n+k_n)\in C_{s_1,\ldots,s_n} $ for some nonnegative integers $k_1,\ldots,k_n$. Clearly, points in $C_{s_1,\ldots,s_n}$ are in bijection with points in $C_{0,0,\ldots,0}$ under the obvious shift. We can find the decomposition of $C_{0,0,\ldots,0}$ with respect to $(k_1,\ldots,k_n)$ as in the proposition above and then shift all the cones in the decomposition by $(s_1,\ldots,s_n)$ to get new cones. This will give us the desired decomposition of $C_{s_1,\ldots,s_n}$. Again, the unique cone of dimension $n$ in this decomposition is $C_{s_1+k_1,\ldots,s_n+k_1}$. Now assume that some coordinates of our cone are fixed, say, we have $C_{s_1,\ldots,s_m,\underline{s_{m+1}}\ldots,\underline{s_n}}$ (without loss of generality, we can assume that fixed coordinates are the last $(n-m)$ coordinates) and $(s_1+k_1,\ldots,s_m+k_m,s_{m+1},\ldots,s_n)\in C_{s_1,\ldots,s_m,\underline{s_{m+1}}\ldots,\underline{s_n}}$. Note that this is an $m$-dimensional cone and points in this cone are in bijection with points in $\mathbb N^m$, in particular, $(s_1+k_1,\ldots,s_m+k_m,s_{m+1},\ldots,s_n)\leftrightarrow(k_1,\ldots,k_m)$. Thus we can find the decomposition of $\mathbb N^m$ with respect to $(k_1,\ldots,k_m)$, then shift all cones by $(s_1\ldots,s_m)$ (this will give us the first $m$ coordinates of each cone) and the last $(n-m)$ coordinates of each cone in this decomposition are $\underline{s_{m+1}},\ldots,\underline{s_n}$. The unique $m$-dimensional cone in this decomposition is $C_{s_1+k_1,\ldots,s_m+k_m,\underline {s_{m+1}},\ldots,\underline{s_n}}$, others have lower dimensions. Therefore, the previous proposition can be restated in a more general context:
\begin{theorem}
\label{th1}
Given any cone $C$ in $\mathbb N^n$ of dimension $k$ and a point $\textbf {a}\in C$, we can decompose $C$ into a disjoint union of finitely many cones, where exactly one cone has dimension $k$ and vertex $\textbf {a}$, and all other cones have strictly lower dimensions.
\end{theorem}

\begin{example}
Let $n=5$ and consider $C_{\underline{5},7,\underline{4},2,\underline{3}}$. Consider $(a_1,\ldots,a_5)=(5,9,4,3,3)\in C_{\underline{5},7,\underline{4},2,\underline{3}}$. The first, the third and the fifth coordinates are fixed once and forever, that is, all cones that we will find have the form $C_{\underline{5},?,\underline{4},?,\underline{3}}$. We are left with the second and the fourth coordinate, that is, $(7,2)$ for the cone and $(9,3)$ for the point. Shifting in the negative direction by $(7,2)$, we will get $(0,0)$ and $(2,1)$ respectively. Thus, it is enough to find the decomposition of $\mathbb N^2$ with respect to $(2,1)$. This is exactly what we did in Example~\ref{ex2}. We obtained $A_{2,1}=\{C_{\underline{0},\underline{0}},C_{\underline{0},1},C_{\underline{1},\underline{0}},C_{\underline{1},1},C_{2,\underline{0}},C_{2,1}\}$. Shifting in the positive direction by $(7,2)$ gives us 
$\{C_{\underline{7},\underline{2}},C_{\underline{7},3},C_{\underline{8},\underline{2}},C_{\underline{8},3},C_{9,\underline{2}},C_{9,3}\}$
and inserting back the first, the third and the fifth coordinates gives us
$$\{C_{\underline{5},\underline{7},\underline{4},\underline{2},\underline{3}},C_{\underline{5},\underline{7},\underline{4},3,\underline{3}},C_{\underline{5},\underline{8},\underline{4},\underline{2},\underline{3}},C_{\underline{5},\underline{8},\underline{4},3,\underline{3}},C_{\underline{5},9,\underline{4},\underline{2},\underline{3}},C_{\underline{5},9,\underline{4},3,\underline{3}}\}.$$

Therefore, $C_{\underline{5},7,\underline{4},2,\underline{3}}$ is a disjoint union of these six cones.
\end{example}

Now we will use these results on monomial ideals. Let $I$ be a good ideal. Then for any vector of nonnegative integers $(a_1,\ldots,a_n)$ we have defined a box $B_{a_1,\ldots,a_n}$ and the corresponding ideal $I_{a_1,\ldots,a_n}$. Clearly, there is a bijection between points in $\mathbb N^n$ and boxes/ideals; recall that if $(a_1,\ldots,a_n)\le(b_1,\ldots,b_n)$, then $I_{a_1,\ldots, a_n}\subseteq I_{b_1,\ldots, b_n}$ by Corollary~\ref{cor}.
\begin{theorem}
\label{th2}
For any good ideal $I$ there exists a finite coloring of $\mathbb N^n$ such that if $(a_1,\ldots,a_n)$ has the same color as $(b_1,\ldots,b_n)$, then $I_{a_1,\ldots,a_n}=I_{b_1,\ldots,b_n}$ and for each color the set of points of this color forms a cone.
\end{theorem}
\begin{proof}
 We use induction on the highest dimension of uncolored cones. We are starting with an $n$-dimensional cone $\mathbb N^n$. We will show how to obtain finitely many cones of strictly lower dimensions, each of which will then be treated similarly in a recursive way. First of all, note that it is possible to find a point $(a_1,\ldots, a_n)$ such that the following holds: if $(b_1,\ldots,b_n)\ge(a_1,\ldots,a_n)$, then $I_{a_1,\ldots,a_n}=I_{b_1,\ldots,b_n}$. Indeed, if we assume the converse, then for every point of $\mathbb N^n$ there exists a strictly larger point that corresponds to a strictly larger ideal, therefore, we can build an infinite chain of strictly increasing ideals, which is impossible, for example, by Noetherianity of the polynomial ring. So existence of such a point $(a_1,\ldots, a_n)$ is justified. Then from the Theorem~\ref{th1}, $\mathbb N^n$ can be covered with a disjoint union of (finitely many) cones in $A_{a_1,\ldots,a_n}$. The unique $n$-dimensional cone in $A_{a_1,\ldots,a_n}$ is $C_{a_1,\ldots,a_n}$ and, as we have just figured out, we may paint all points in this cone with the same color. Now we are left with a finite disjoint union of cones of dimensions at most $n-1$ which need to be painted and we apply induction on each of them, lowering the maximal dimension by 1 again. Since it is a finite process, in the end we will obtain a finite coloring of $\mathbb N^n$.
\end{proof}

We remark that the coloring described above is not unique since it depends on the choice of $(a_1,\ldots, a_n)$ and its lower dimensional analogues. 

\begin{example}
Let $I$ be the ideal in Example~\ref{ex}. We can choose $(a_1,a_2)=(1,1)$ since $I_{b_1,b_2}=I_{1,1}$ for all $(b_1,b_2)\ge (1,1)$. Then $\mathbb N^2$ is a disjoint union of $\textcolor{red}{C_{1,1}}$, $C_{\underline{0},1}$, $C_{1,\underline{0}}$  and $\textcolor{blue}{C_{\underline{0},\underline{0}}}$. Now consider $C_{\underline{0},1}$. We see that $I_{0,b}=I_{0,2}$ for all $b\ge 2$. Therefore, we consider the decomposition of $C_{\underline{0},1}$ with respect to $(0,2)$: $C_{\underline{0},1}$ is a disjoint union of $\textcolor{green}{C_{\underline{0},2}}$ and $\textcolor{cyan}{C_{\underline{0},\underline{1}}}$. Similarly, $C_{1,\underline{0}}$ is a disjoint union of $\textcolor{orange}{C_{2,\underline{0}}}$ and $\textcolor{magenta}{C_{\underline{1},\underline{0}}}$. The left picture in Figure~3 describes the coloring we have just discussed. The picture on the right describes another possible coloring if, for instance, we choose $(a_1,a_2)=(0,2)$.

\begin{figure}[H]
\centering
\begin{tikzpicture}
	\filldraw[fill=blue]    (0,0) rectangle (1,1);
	\filldraw[fill=magenta] (1,0) rectangle (2,1);
	\filldraw[fill=orange]  (2,0) rectangle (3,1);
	\filldraw[fill=orange]  (3,0) rectangle (4,1);

	\filldraw[fill=cyan]    (0,1) rectangle (1,2);
	\filldraw[fill=red]     (1,1) rectangle (2,2);
	\filldraw[fill=red]     (2,1) rectangle (3,2);
	\filldraw[fill=red]     (3,1) rectangle (4,2);

	\filldraw[fill=green]   (0,2) rectangle (1,3);
	\filldraw[fill=red]     (1,2) rectangle (2,3);
	\filldraw[fill=red]     (2,2) rectangle (3,3);
	\filldraw[fill=red]     (3,2) rectangle (4,3);

	\filldraw[fill=green]   (0,3) rectangle (1,4);
	\filldraw[fill=red]     (1,3) rectangle (2,4);
	\filldraw[fill=red]     (2,3) rectangle (3,4);
	\filldraw[fill=red]     (3,3) rectangle (4,4);
	\foreach \x in {0,...,4}{
    \node [below, very thin] at (\x,0) {$_\x$};
    \node [left, very thin] at (0,\x) {$_\x$};
    }
    
\end{tikzpicture}
\hspace{2cm}
\begin{tikzpicture}
	\filldraw[fill=magenta] (0,0) rectangle (1,1);
	\filldraw[fill=blue] (1,0) rectangle (2,1);
	\filldraw[fill=cyan] (2,0) rectangle (3,1);
	\filldraw[fill=cyan] (3,0) rectangle (4,1);

	\filldraw[fill=green] (0,1) rectangle (1,2);
	\filldraw[fill=red] (1,1) rectangle (2,2);
	\filldraw[fill=red] (2,1) rectangle (3,2);
	\filldraw[fill=red] (3,1) rectangle (4,2);

	\filldraw[fill=orange] (0,2) rectangle (1,3);
	\filldraw[fill=orange] (1,2) rectangle (2,3);
	\filldraw[fill=orange] (2,2) rectangle (3,3);
	\filldraw[fill=orange] (3,2) rectangle (4,3);

	\filldraw[fill=orange] (0,3) rectangle (1,4);
	\filldraw[fill=orange] (1,3) rectangle (2,4);
	\filldraw[fill=orange] (2,3) rectangle (3,4);
	\filldraw[fill=orange] (3,3) rectangle (4,4);
	 \foreach \x in {0,...,4}{
    \node [below, very thin] at (\x,0) {$_\x$};
    }
     \foreach \x in {0,...,4}{
    \node [left, very thin] at (0,\x) {$_\x$};
    }
\end{tikzpicture}
\caption{two examples of possible colorings of $\mathbb{N}^2$, associated to $I$} 
\end{figure}
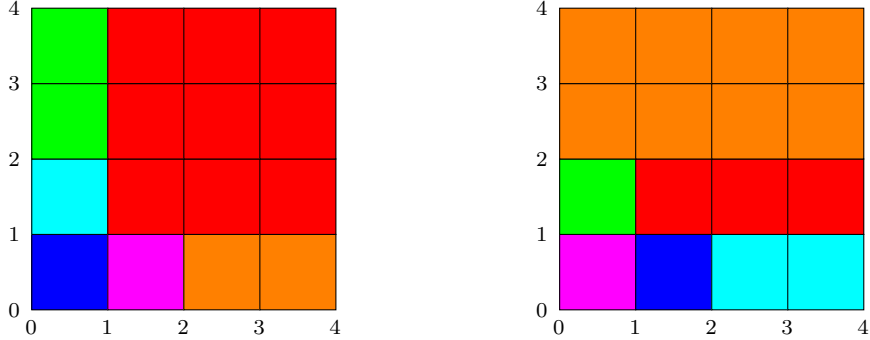

\end{example}

Given a good ideal $I$, any coloring as in Theorem~\ref{th2} represents a finite disjoint union of cones. Each cone has a vertex. Let $L$ denote the maximum of sums of coordinates of these vertices. This number depends on $I$ and on the coloring we choose, but we will not put any additional indices: as soon as we found some coloring (which exists according to Theorem~\ref{th2}), we simply work with it once and forever. For example, for both colorings in Figure~3 we have $L=2$. 
\begin{remark}
\label{rem2}
Note that from the construction in Theorem~\ref{th2} it is clear that $L$ can not be attained at a zero dimensional cone, in other words, for any zero dimensional cone, the sum of coordinates of its vertex is strictly less than $L$.
\end{remark}
The geometric meaning of this number is the following: starting from $I^{L+1}$, powers of $I$ look similar to each other in some sense. For instance, for the left coloring in Figure~3 we know that every power of $I$ starting from $I^3$ consists of a green box, an orange box and several red boxes and we exactly know where each of them is. This means, there is a pattern on high powers of $I$, and this is a key point for finding the Ratliff--Rush closure of $I$.

\section{The main result}
Now we are ready to prove our main theorem, but first we need a preliminary lemma.
\begin{lemma}
\label{lem1}
Let $I$ be a good ideal and let $Q$ be any nonnegative integer. Then there exists a number $L(Q)$ such that for any $l\ge L(Q)$ the following holds: for every minimal generator $m$ of $I^l$ there is an $i$ such that $m=m'\mu_i^Q$ and $m'$ is a minimal generator of $I^{l-Q}$.
\end{lemma}
\begin{proof} 
If $Q=0$, the claim is trivial. Let $Q>0$ and let $L$ be the number defined in the end of Section~5. Take $L(Q)=L+nQ-n+2$ and let $l\ge L(Q)$. Let $m$ be a minimal generator of $I^l$, then it belongs to some box $B_{b_1,\ldots, b_n}$ with $b_1+\ldots+b_n=l-1\ge L+nQ-n+1$. We also know that $(b_1,\ldots, b_n)$ belongs to one of the cones from our coloring; assume that the vertex of this cone is $(a_1,\ldots,a_n)$ (some coordinates are underlined, some are not underlined). Now we want to find a coordinate $b_i$ such that $(b_1,\ldots, b_{i-1}, b_i-Q, b_{i+1},\ldots, b_n)$ belongs to the same cone. Assume that it is not possible. Then it follows that $b_1-Q\le a_1-1, \ldots, b_n-Q\le a_n-1$. These inequalities yield a contradiction $L<b_1+\ldots+b_n-nQ+n\le a_1+\ldots+a_n\le L$, where the last inequality follows from the definition of $L$. So we can find an index $i$ such that $b_i-Q\ge a_i$ (in particular, this implies that $a_i$ is not underlined). Without loss of generality we assume that $i=1$. That means, $(b_1,\ldots, b_n)$ and $(b_1-Q,b_2,\ldots, b_n)$ are both in the same cone. This implies that their colors are equal, which means $I_{b_1,\ldots, b_n}=I_{b_1-Q,b_2,\ldots,b_n}$. In other words, the set of monomials in $B_{b_1,\ldots, b_n}\cap G(I^{l})$ coincides with the set of monomials in $B_{b_1-Q,b_2,\ldots,b_n}\cap G(I^{l-Q})$ up to a shift by $\mu_1^Q$.
Therefore, if $m \in B_{b_1,\ldots, b_n}$ is a minimal generator of $I^l$, then $\frac{m}{\mu_1^Q}\in B_{b_1-Q,b_2,\ldots,b_n}$ is a minimal generator of $I^{l-Q}$, as desired.
\end{proof}

Now let us consider the following line of boxes which is in bijection with nonnegative integer points on the $x_1$-axis: $B_{0,0,\ldots,0}$, $B_{1,0,\ldots,0}$, $B_{2,0,\ldots,0}$ etc. Let $B_{q_1,0\ldots,0}$ be the stabilizing box of this sequence in a sense that $q_1$ is the smallest nonnegative integer such that $I_{t,0,\ldots,0}=I_{q_1,0,\ldots,0}$ for all $t\ge q_1$. Similarly, considering lines of boxes going along the other coordinate axes, we will get $q_2,q_3,\ldots, q_n$. Denote $q:=\max\{q_1,\ldots, q_n\}$.
\begin{theorem}
Let $I$ be a good ideal, let $L$, $q_i$ and $q$ be as above.  Then $\tilde{I}=I_{q_1,0,\ldots,0}\cap I_{0,q_2,\ldots,0}\cap\ldots\cap I_{0,\ldots,0,q_n}$.
\end{theorem}
\begin{proof}
 $\subseteq$ Let $l\ge q$. We will show that $I^{l+1}:I^l\subseteq I_{q_1,0,\ldots,0}\cap I_{0,q_2,\ldots,0}\cap\ldots\cap I_{0,\ldots,0,q_n}$. In fact, we will show that $I^{l+1}:I^l\subseteq I_{q_1,0,\ldots,0}$, other inclusions are analogous. Since $I^{l+1}:I^l\subseteq I^{l+1}:\langle\mu_1^l\rangle$, it is sufficient to show that $I^{l+1}:\langle \mu_1^l\rangle\subseteq I_{q_1,0,\ldots,0}$. By Proposition \ref{prop1}, $I^{l+1}:\langle \mu_1^l\rangle=I_{l,0,\ldots,0}$ which equals $I_{q_1,0,\ldots,0}$, given the way $I_{q_1,0,\ldots,0}$ was defined and given that $l\ge q\ge q_1$ . Therefore, everything follows.

 $\supseteq$ Let $m\in I_{q_1,\ldots,0}\cap I_{0,q_2,\ldots,0}\cap\ldots\cap I_{0,\ldots,0,q_n}$, let $l\ge L(q)=L+nq-n+2$ (as in Lemma~\ref{lem1}). We will show that for every $m_l\in I^l$ we have $mm_l\in I^{l+1}$. It is enough to consider $m_l$ to be minimal generators of $I^l$. First of all, from Lemma~\ref{lem1} we know that we can factor out some $\mu_i^q$ from $m_l$ and get a minimal generator of $I^{l-q}$, that is, $m_l=\mu_i^qm_{l-q}$ for some index $i$ and $m_{l-q}$ a minimal generator in $I^{l-q}$. Also, since $m$ belongs (in particular) to  $I_{0,\ldots,0,q_i,0\ldots,0}=I^{q_i+1}:\langle\mu_i^{q_i}\rangle$, it means, $m\mu_i^{q_i}\in I^{q_i+1}$.

 Therefore, $mm_l=m\mu_i^{q_i}\mu_i^{q-q_i}m_{l-q}\in I^{l+1}$ since $m\mu_i^{q_i}\in I^{q_i+1}$, $\mu_i^{q-q_i}\in I^{q-q_i}$, $m_{l-q}\in I^{l-q}$.
\end{proof}

\section{Explicit computation of $I_{0,\ldots,0,q_i,0,\ldots,0}$}
 We have seen that, given a good ideal $I$, its Ratliff--Rush closure is computed as $\tilde{I}=I_{q_1,0,\ldots,0}\cap I_{0,q_2,\ldots,0}\cap\ldots\cap I_{0,\ldots,0,q_n}$. Surprisingly, no other boxes affect the Ratliff--Rush closure, but only those going along coordinate axes. Therefore, we would like to know more about $I_{0,\ldots,0,q_i,0\ldots,0}$. Let $i=1$, other cases are analogous. So far we only know that $I_{t,0,\ldots,0\ldots,0}=I^{t+1}:\langle\mu_1^t\rangle$. Computation of $I^{t}$ might take much time if $t$ is large enough. In addition, we do not know yet at which moment the line has stabilized. So far the process seems more complicated than it is. We will state a few results that will make this computation easier.
\begin{lemma}
\label{rem1}
 If $I$  is a good ideal, then $I_{t+1,0,\ldots,0}=(I_{t,0\ldots,0}\cdot I):\langle\mu_1\rangle$ for all $t\ge 0$.
\end{lemma}

\begin{proof}
 $\subseteq$ Let $m\in I_{t+1,0,\ldots,0}=I^{t+2}:\langle \mu_1^{t+1}\rangle$. Then $m\mu_1^{t+1}=cg_1\cdots g_{t+2}\in I^{t+2}$, where $g_1,\ldots,g_{t+2}$ are minimal generators of $I$ and $c$ is some monomial. We claim that $cg_1\cdots g_{t+1}$ is divisible by $\mu_1^{t}$. Indeed, if it is not the case, then the $x_1$-exponent of $g_{t+2}$ is strictly greater than $d_1$, which is impossible since $g_{t+2}$ is a minimal generator of $I$. Therefore, $m\mu_1=\frac{cg_1\cdots g_{t+1}}{\mu_1^{t}}g_{t+2}$. Since $\frac{cg_1\cdots g_{t+1}}{\mu_1^{t}}\in I_{t,0,\ldots,0}$ and $g_{t+2}\in I$, we are done.

 $\supseteq$ Let $m\in(I_{t,0\ldots,0}\cdot I):\langle\mu_1\rangle$. Then $m\mu_1=gm_1$, where $g\in I$, $m_1\in I_{t,0\ldots,0}$. Then $m\mu_1^{t+1}=m\mu_1\mu_1^{t}=gm_1\mu_1^{t}\in I^{t+2}$ since $g\in I$ and $m_1\mu_1^{t}\in I^{t+1}$.
 \end{proof}
 \begin{remark}
We would like to point out that if $I_{t,0,\ldots,0}=I_{t+1,0,\ldots,0}$, then the line has stabilized, that is, $I_{k,0,\ldots,0}=I_{t,0,\ldots,0}$ for all $k\ge t$. This is a direct corollary of Lemma~\ref{rem1}.
 \end{remark}
 \begin{remark}
For all $t\ge 0$ let $E_t:=G(I_{t,0,\ldots,0})$ and for all $t\ge 1$ let
$$F_{t}:=\{m \in E_t=G(I_{t,0\ldots,0})\mid m\not\in I_{t-1,0,\ldots,0}\}.$$
We also set $E_{-1}=\emptyset$, $F_0=G(I)$. We will show how, given $\{E_{t-1}, F_{t}\}$, one can obtain $\{E_t, F_{t+1}\}$ for any $t\ge 0$.
Clearly, $E_{t}$ is the reduced union of $E_{t-1}$ and $F_{t}$.  From Lemma~\ref{rem1} we remember that
$I_{t+1,0,\ldots,0}=(I_{t,0,\ldots,0}\cdot I):\langle\mu_1\rangle=(\langle E_{t-1} \cup F_{t}\rangle\cdot I):\langle\mu_1\rangle=((I_{t-1,\ldots,0}+\langle F_{t}\rangle)\cdot I):\langle\mu_1\rangle=(I_{t-1,\ldots,0}\cdot I+\langle F_{t}\rangle\cdot I):\langle\mu_1\rangle=(I_{t-1,\ldots,0}\cdot I):\langle\mu_1\rangle+(\langle F_{t}\rangle\cdot I):\langle\mu_1\rangle=I_{t,0,\ldots,0}+(\langle F_{t}\rangle\cdot I):\langle\mu_1\rangle$. Therefore, we conclude that minimal generators of $I_{t+1,0,\ldots,0}$ which are not in $I_{t,0,\ldots,0}$  (our future $F_{t+1}$) could only be among the minimal generators of $(\langle F_{t}\rangle\cdot I):\langle\mu_1\rangle$, that is, only new monomials from the previous iteration can give rise to new monomials in the next iteration. Therefore, in order to compute $F_{t+1}$ we need to compute $A_t:=\left\{\frac{fm}{\gcd (fm,\mu_1)} \mid f\in F_t, m\in G(I)\right\}$, reduce this set and throw away monomials that are already in $\langle E_{t}\rangle=I_{t,0,\ldots,0}$. If $F_{t+1}=\emptyset$, it means that $I_{t+1,0\ldots,0}=I_{t,0\ldots,0}$ and $E_{t}$ is the desired generating set.
 \end{remark}
\begin{remark}
Now we know that in order to compute $F_{t+1}$  we need to compute $A_t=\left\{\frac{fm}{\gcd (fm,\mu_1)} \mid f\in F_t, m\in G(I)\right\}$, reduce this set and throw away monomials that are already in $\langle E_{t}\rangle=I_{t,0,\ldots,0}$, where $E_t$ has already been computed. This is already quite straightforward, but we can simplify the calculations a bit more. First of all note that if $m=\mu_1$, then for any $f\in F_t$ we have $\frac{fm}{\gcd(fm,\mu_1)}=f\in \langle F_t\rangle \subseteq \langle E_t\rangle$ and thus this monomial will be thrown away anyway. If $m=\mu_i$ for some $2\le i\le n$, say, $m=\mu_2$, then  for any $f\in F_t$ we get that $\frac{fm}{\gcd(fm,\mu_1)}$ is divisible by $\mu_2\in \langle E_t\rangle$ (recall that all $\mu_i$ are minimal generators of all $I_{a_1,\ldots,a_n}$) and will thus be thrown away. In particular, this implies that in the definition of $A_t$ one can replace $G(I)$ with $P(I):=G(I)\backslash\{\mu_1,\ldots,\mu_n\}$. Another observation is the following. Assume that $f\in F_t$, $m\in P(I)$. Write $fm=x_1^{\alpha_1}\cdots x_n^{\alpha_n}$. If $\alpha_i\ge d_i$ for some $2\le i \le n$, say, $\alpha_2\ge d_2$, then $\frac{fm}{\gcd (fm,\mu_1)}$ is a again a multiple of $\mu_2\in \langle E_t\rangle)$. Therefore, in the definition of $A_t$ we may force that $\mathrm{deg}_{x_i}(fm)<d_i$ for all $2\le i\le n$, that is, we may take $A_t:=\{\frac{fm}{\gcd(fm,\mu_1)} \mid f\in F_t, m\in P(I), \mathrm{deg}_{x_i}(fm)<d_i, 2\le i\le n\}$. Finally, consider $x_1^{\alpha_1}\cdots x_n^{\alpha_n}\mu_1^t=f\mu_1^tm\in I^{t+2}$ since $f\mu_1^t\in I^{t+1}$ according to Proposition~\ref{prop1}. Since we have imposed $\alpha_i<d_i$ for all $2\le i\le n$, we must have $\alpha_1\ge d_1$, since otherwise $f\mu_1^tm$ would belong to a box with the sum of coordinates at most $t$. Thus $\gcd(fm,\mu_1)=\mu_1$. Therefore, we conclude that we can write
$A_t=\left\{\frac{fm}{\mu_1} \mid f\in F_t, m\in P(I), \mathrm{deg}_{x_i}(fm)<d_i, 2\le i\le n\right\}$. In order to compute $F_{t+1}$ one still needs to reduce this set and throw away monomials which are already in $\langle E_t\rangle$, if needed.

\end{remark}
The algorithm below produces $E=G(I_{q_1,0,\ldots,0})$ given the input $G(I)$.

 \begin{algorithm}[H]
 \Begin{
 $E:=\emptyset$, $F:=G(I)$;
 
 \While{$F\not=\emptyset$}{
  $E:=\text{reduce}(E\cup F)$\;

  $F_1:=\emptyset$\;
  \For{$m\in P(I)=G(I)\backslash\{\mu_1,\ldots,\mu_n\}$, $f\in F$}
  {
  
  \If{$\mathrm{deg}_{x_i}(fm)<d_i$ {\rm for all} $2\le i\le n$}
  {
   $a:=\frac{fm}{\mu_1}$\;
  \If{$a\not\in \langle E\rangle$}
  {
   $F_1:=F_1\cup \{a\}$\;
 
   }
 
   }
  
  }
   $F:=\text{reduce}(F_1)$\;
   
  }
  print $E$\;
 }
 \end{algorithm}

\section{Examples}
 \begin{example}
Let $R=\mathbb K[x,y,z]$ and let 
$$I=\langle \mu_1,\mu_2,\mu_3,m_1,m_2,m_3\rangle=\langle x^{29}, y^{29}, z^{29}, x^{28}y^8z^8, x^8y^{28}z^8, x^8y^8z^{28}\rangle\subset R.$$
Since $I$ satisfies the sufficient condition, it is a good ideal. Computations in Singular show that $$I^2:I=I+\langle x^{27}y^{27}z^{27}\rangle,$$
 $$I^3:I^2=I^4:I^3=I+\langle x^{26}y^{27}z^{27}, x^{27}y^{26}z^{27}, x^{27}y^{27}z^{26}\rangle,$$
 $$I^5:I^4=I^6:I^5=\cdots=I^{10}:I^9=I+\langle x^{26}y^{26}z^{26}\rangle.$$
 It is natural to conjecture that $\tilde{I}=I+\langle x^{26}y^{26}z^{26}\rangle.$
 Now let us see what we get if we apply the algorithm above. We start with $E_{-1}=\emptyset$, $F_0=G(I)$. Then we obtain $E_0$ by reducing $E_{-1}\cup F_0$, that is, $E_0=G(I)$ (as it should be). In order to compute $F_1$, we take all products of $F_0=G(I)$ with $P(I)$, keeping in mind that $y-$ and $z-$ coordinates of each product need to be less than 29, and divide each such product by $\mu_1$. If we take $f=\mu_1\in F_0=G(I)$ and any $m\in P(I)$, then $\frac{fm}{\mu_1}=m$  will be thrown away (and this always happens in the first iteration, but never afterwards since in the higher iterations $\mu_1$ never belongs to any of the considered sets). The only monomial that is not thrown away is $\frac{m_1^2}{\mu_1}=\frac{x^{56}y^{16}z^{16}}{x^{29}}=x^{27}y^{16}z^{16}$. This monomial is not in $\langle E_0\rangle$, therefore, we add it to our set $F_1$ (and this set is already reduced). Therefore, $E_{0}=G(I)$, $F_1=\{x^{27}y^{16}z^{16}\}$. Now $E_1=E_0\cup F_1=G(I)\cup \{x^{27}y^{16}z^{16}\}$ (this union is already reduced), and in order to compute $F_2$ we need to multiply $x^{27}y^{16}z^{16}$ with monomials from $P(I)$ (keeping in mind the condition on $y-$ and $z-$ coordinates) and divide the products by $\mu_1$. The only possible monomial is $\frac{x^{27}y^{16}z^{16}\cdot m_1}{\mu_1}=x^{26}y^{24}z^{24}$. This monomial is not in $\langle E_1\rangle$, therefore, $F_2=\{x^{26}y^{24}z^{24}\}$. $E_2=E_1\cup F_2=G(I)\cup \{x^{27}y^{16}z^{16},x^{26}y^{24}z^{24}\}$ (this set is already reduced) and if we try to compute $F_3$, we see that we can not get any new monomials. Therefore, $F_3=\emptyset$ and the stabilizing point is $I_{2,0,0}=\langle E_2\rangle=I+\langle x^{27}y^{16}z^{16},x^{26}y^{24}z^{24}\rangle$

 By symmetry,
 $$I_{0,2,0}=I+\langle x^{16}y^{27}z^{16}, x^{24}y^{26}z^{24}\rangle$$ and $$I_{0,0,2}=I+\langle x^{16}y^{16}z^{27}, x^{24}y^{24}z^{26}\rangle.$$ According to the theorem, $\tilde{I}=I_{2,0,0}\cap I_{0,2,0}\cap I_{0,0,2}=I+\langle x^{26}y^{26}z^{26}\rangle$, just as expected.
 \end{example}
 \begin{example}

Consider $I=\langle x^{53},y^{56},z^{59},w^{61},x^{50}y^{18}z^{20}w^{25},x^{15}y^{54}z^{22}w^{24},\\ x^{18}y^{20}z^{56}w^{22},  x^{16}y^{19}z^{23}w^{60}\rangle\subset\mathbb K[x,y,z,w]$.
 
 Since $I$ satisfies the sufficient condition, it is a good ideal.
 Then, applying the algorithm from Section~7, we will get
 $$I_{1,0,0,0}=I+\langle x^{47}y^{36}z^{40}w^{50}\rangle=I_{2,0,0,0}=I_{3,0,0,0}=I_{4,0,0,0}=\cdots$$
 $$I_{0,1,0,0}=I+\langle x^{30}y^{52}z^{44}w^{48}\rangle=I_{0,2,0,0}=I_{0,3,0,0}=I_{0,4,0,0}=\cdots$$
 $$I_{0,0,1,0}=I+\langle x^{36}y^{40}z^{53}w^{44}\rangle=I_{0,0,2,0}=I_{0,0,3,0}=I_{0,0,4,0}=\cdots$$
 $$I_{0,0,0,1}=I+\langle x^{32}y^{38}z^{46}w^{59}\rangle=I_{0,0,0,2}=I_{0,0,0,3}=I_{0,0,0,4}=\cdots$$
Then $\tilde{I}=I_{1,0,0,0}\cap I_{0,1,0,0}\cap I_{0,0,1,0}\cap I_{0,0,0,1}=I+\langle x^{47}y^{52}z^{53}w^{59}\rangle$ which coincides with our expectations based on computations in Singular.
 \end{example}
\begin{example}
 Let $I=\langle x^{41},y^{41},z^{41},x^{40}y^5z^5,x^5y^{40}z^5,x^5y^5z^{40}\rangle\subset\mathbb K[x,y,z]$. It will be proven in Example~\ref{ex1} that $I$ is a good ideal. All the new monomials can only be obtained from powers of non-corners:
 $$I_{1,0,0}=I+\langle x^{39}y^{10}z^{10}\rangle,$$
 $$I_{2,0,0}=I_{1,0,0}+\langle x^{38}y^{15}z^{15}\rangle,$$
 $$\ldots$$
 $$I_{6,0,0}=I_{5,0,0}+\langle x^{34}y^{35}z^{35}\rangle.$$
 Here the line stabilizes. We similarly get $I_{0,6,0}$ and $I_{0,0,6}$. Each of these three ideals has twelve minimal generators in total. Intersecting them we will get $$\tilde I=I_{6,0,0}\cap I_{0,6,0}\cap I_{0,0,6}=I+\langle x^{34}y^{35}z^{35},x^{35}y^{34}z^{35},x^{35}y^{35}z^{34}\rangle.$$ If we compute successive quotients via computer algebra, the result is the following:
$I^2:I^1$ has 7 minimal generators, that is, $|G(I^2:I^1)|=7$; $|G(I^3:I^2)|=9$; $|G(I^4:I^3)|=12$; $|G(I^5:I^4)|=16$; $|G(I^6:I^5)|=21$; $|G(I^7:I^6)|=27$; $|G(I^8:I^7)|=31$; $|G(I^9:I^8)|=33$; $|G(I^{10}:I^9)|=33$; $|G(I^{11}:I^{10})|=31$; $|G(I^{12}:I^{11})|=24$; $|G(I^{13}:I^{12})|=18$; $|G(I^{14}:I^{13})|=13$; $|G(I^{15}:I^{14})|=9$.  
$I^{15}:I^{14}$ finally coincides with the ideal obtained above (but we still can not be sure this is the Ratliff--Rush closure of $I$ assuming that we are not using our formula). It takes much time to perform these computations using computer algebra, whereas the computation of $I_{6,0,0}$, $I_{0,6,0}$ and $I_{0,0,6}$ and their intersection is much easier and can even be done by hand in this example.

 \end{example}
 \section {A bit more about good and bad ideals}
Let $I$ be an ideal that satisfies the necessary condition, but does not satisfy the sufficient condition. How could we possibly figure out whether it is a good or a bad ideal? We will start with ideals that have only one extra generator except $\mu_i$. Let $I=\langle \mu_1,\ldots, \mu_n, m\rangle$, where $\mu_i=x_i^{d_i}$ and $m=x_1^{\alpha_1}\cdots x_n^{\alpha_n}$ with $\frac{\alpha_1}{d_1}+\ldots+\frac{\alpha_n}{d_n}\ge 1$. Then the minimal generators of $I^l$ will be of the form $\mu_1^{k_1}\cdots \mu_n^{k_n}m^k$, where $k_1+\ldots+k_n+k=l$. We will first of all show that if $d=\lcm(d_1,\ldots,d_n)$, then $m^d$ is not needed as a minimal generator of $I^d$. Indeed, $m^d=x_1^{d\alpha_1}\cdots x_n^{d\alpha_n}=\mu_1^{\frac{d\alpha_1}{d_1}}\cdots \mu_n^{\frac{d\alpha_n}{d_n}}$ and $\frac{d\alpha_1}{d_1}+\ldots+\frac{d\alpha_n}{d_n}=d(\frac{\alpha_1}{d_1}+\ldots+\frac{\alpha_n}{d_n})\ge d$. Therefore, $m^d\in \langle \mu_1,\ldots, \mu_n\rangle^d$. Let $K$ be the smallest positive integer such that $m^K\in \langle \mu_1,\ldots, \mu_n\rangle^K$. As we have seen, $d=\lcm (d_1,\ldots, d_n)$ is an upper bound for $K$ (but this upper bound can be easily improved). Therefore, for any $l$, the minimal generators of $I^l$ will be of the form $\mu_1^{k_1}\cdots \mu_n^{k_n}m^k$, where $k_1+\ldots+k_n+k=l$ and $k\le K-1$. Conversely, each monomial of this form is a minimal generator of $I^l$. Clearly, such a monomial belongs to a box whose sum of coordinates is $l-1$ if and only if $m^k$ belongs to a box whose sum of coordinates is $k-1$ (if $k=0$, the monomial is just a corner and this case is trivial). Therefore, the conclusion is the following: for all $k\le K-1$ 
we need to check if $m^k$ belongs to a box whose sum of coordinates is $k-1$. If the answer is positive, $I$ is good and otherwise it is bad.
\begin{example}
Let $I=\langle x^{10}, y^{10}, z^{10}, x^2y^2z^8\rangle\subset\mathbb K[x,y,z]$. Let $m=x^2y^2z^8$. Then $m^5\in\langle\mu_1,\mu_2,\mu_3\rangle^5$, but $m^4\not\in\langle\mu_1,\mu_2,\mu_3\rangle^4$, thus $K=5$. Further, $m^2=x^4y^4z^{16}\in B_{0,0,1}$, $m^3=x^6y^6z^{24}\in B_{0,0,2}$ and $m^4=x^8y^8z^{32}\in B_{0,0,3}$. Therefore, $I$ is a good ideal.
\end{example}
The situation does not change much if $I$ has more generators. 

Let $I=\langle \mu_1,\ldots, \mu_n, m_1,\ldots,m_t\rangle$. Let $K_i$ be the smallest positive integer such that $m_i^{K_i}\in \langle \mu_1,\ldots, \mu_n\rangle^{K_i}$, $1\le i\le t$. Then minimal generators of $I^l$ will be of the form $\mu_1^{k_1}\cdots \mu_n^{k_n}m_1^{j_1}\cdots m_t^{j_t}$, where $k_1+\ldots+k_n+j_1+\ldots+j_t=l$ and $j_i\le K_i-1$, $1\le i\le t$. Note that the converse is not true: unlike in the case with only one additional generator, not every monomial of this form is a minimal generator. Therefore, we want to check whether each of them belongs to a box with the sum of coordinates \emph{at least} $l-1$. As before, such a monomial belongs to a box whose sum of coordinates is at least $l-1$ if and only if $m_1^{j_1}\cdots m_t^{j_t}$ belongs to a box whose sum of coordinates is at least $j_1+\ldots+j_t-1$ (if $j_1=\ldots=j_t=0$, the monomial is just a corner and this case is trivial). Therefore, it is enough to check whether all these monomials belong to boxes with the sums of coordinates greater or equal to the expected ones. If the answer is positive, the ideal is good, otherwise it is bad. Finding $K_i$ requires knowing the interaction of $m_i$ with $\mu_1,\ldots,\mu_n$. We could have defined $K_i$ in a different way (and thus make them smaller), for instance,
$$K_i=\min\{K\in \mathbb{N} \mid m_i^{K}\in \langle \mu_1, \ldots, \mu_n, m_1,\ldots,m_{i-1}, \hat{m}_i, m_{i+1},\ldots, m_t\rangle^{K}\}$$
($\hat{m}_i$ means that $m_i$ is omitted), but this would require a bit more justification in general.
\begin{example}
Let $I=\langle x^5,y^5,z^5,x^2y^4z,x^4y^2z\rangle\subset\mathbb K[x,y,z]$. Then $K_1=3$ since $m_1^3=x^6y^{12}z^3$ is divisible by $\mu_1\mu_2^2$, but $m_1^2=x^4y^8z^2\not\in\langle\mu_1,\mu_2,\mu_3\rangle^2$. The same holds for $m_2$, so $K_2=3$. Therefore, we would like to check all the monomials $m_1^{j_1}m_2^{j_2}$ with $0\le j_1\le j_2\le 2$ (we might take $j_1\le j_1$ due to the symmetry of the ideal). The case $j_1=j_2=0$ is not to be considered, as discussed earlier. If $j_1=0,j_2=1$, then $m_2\in B_{0,0,0}$. If $j_1=0, j_2=2$, then $m_2^2=x^8y^4z^2\in B_{1,0,0}$. Note, however, that this monomial is not a minimal generator since it is strictly divisible by $m_1\mu_1=x^7y^4z$. If $j_1=j_2=1$, then $m_1m_2=x^6y^6z^2\in B_{1,1,0}$. This is the only box where this monomial can be placed. In particular, this implies it is not a minimal generator of $I^2$. Indeed, $x^6y^6z^2$ is strictly divisible by by $\mu_1\mu_2$. The last case to consider is $j_1=1,j_2=2$, but then $m_1m_2^2$ can not be a minimal generator of $I^3$ since $m_1m_2$ was not a minimal generator of $I^2$, as we have already seen. Now, for any $l\ge 1$, minimal generators of $I^l$ (up to shifts by corners) can not be of types other than these. Therefore, $I$ is a good ideal. Note that we could have chosen $K_1$ and $K_2$ in a less naive way, as discussed before this example: $K_1=K_2=2$ since $m_1^2$ is divisible by $\mu_2m_2$ and $m_2^2$ is divisible by $\mu_1m_1$. Then the cases to consider would be $(j_1,j_2)=(0,1)$ and $(j_1,j_2)=(1,1)$ (given the symmetry of the ideal).
\end{example}
\begin{example}
\label{ex1}
Let $I=\langle x^{41},y^{41},z^{41},x^{40}y^5z^5,x^5y^{40}z^5,x^5y^5z^{40}\rangle\subset\mathbb K[x,y,z]$. It is not hard to notice that $K_1=K_2=K_3=9$. Also we notice that $m_1m_2$, $m_2m_3$ and $m_1m_3$ are not minimal generators of $I^2$. Thus the only monomials we need to check are powers of $m_i$. We will check powers of $m_1$, others are analogous. We see that $m_1^2=x^{80}y^{10}z^{10}\in B_{1,0,0}$, $m_1^3=x^{120}y^{15}z^{15}\in B_{2,0,0}$, $m_1^4=x^{160}y^{20}z^{20}\in B_{3,0,0}$, $m_1^5=x^{200}y^{25}z^{25}\in B_{4,0,0}$, $m_1^6=x^{240}y^{30}z^{30}\in B_{5,0,0}$, $m_1^7=x^{280}y^{35}z^{35}\in B_{6,0,0}$ and $m_1^8=x^{320}y^{40}z^{40}\in B_{7,0,0}$. Therefore, $I$ is a good ideal. Note that $m_1^8\not\in G(I^8)$ since it is divisible by $\mu_1^7\cdot m_2$ and  $\mu_1^7\cdot m_3$. Again, a less naive choice of $K_i$ would result into $K_1=K_2=K_3=8$ and we would not need to consider the case $(j_1,j_2,j_3)=(8,0,0)$.

\end{example}
 \section{Powers of good ideals}
A natural question to ask is whether powers of good ideals also good. The answer is clearly positive if $n=1$, but in most of the other cases the answer is negative. 

Let $I$ be a good ideal. As before, for nonnegative integers $a_1,\ldots,a_n$, by $B_{a_1,\ldots,a_n}$ we denote the corresponding box, associated to $I$. But $I^k$ will determine its own boxes: $x_1^{kd_1},\ldots,x_n^{kd_n}$ are minimal generators of $I^k$, thus the new boxes will have sizes $kd_1,\ldots,kd_n$. For nonnegative integers $b_1,\ldots, b_n$  we denote by $B^k_{b_1,\ldots,b_n}$ the corresponding box, associated to $I^k$, that is, $B^k_{b_1,\ldots,b_n}=([b_1kd_1,(b_1+1)kd_1]\times[b_2kd_2,(b_2+1)kd_2]\times\ldots\times[b_nkd_n,(b_n+1)kd_n])\cap \mathbb{N}^n$.

\begin{proposition}
\label{prop}
Let $I$ be a good ideal and $a_1,\ldots, a_n$ be nonnegative integers. For $k\ge 1$, let $b_1=\lfloor\frac{a_1}{k}\rfloor$, \ldots, $b_n=\lfloor\frac{a_n}{k}\rfloor$. Then $B_{a_1,\ldots,a_n}\subseteq B^k_{b_1,\ldots,b_n}$.
\end{proposition}
\begin{proof}
Clearly, $b_ik\le a_i$ and $a_i+1\le(b_i+1)k$ for all $i\in\{1,\ldots,n\}$. Therefore, $b_ikd_i\le a_id_i<(a_i+1)d_i\le(b_i+1)kd_i$, as desired.
\end{proof}

\begin{proposition}
\label{good2}
Let $I\in \mathbb K[x,y]$ be a good ideal. Then all powers of $I$ are also good.
\end{proposition}
\begin{proof}
Let $x^{d_1}$ and $y^{d_2}$ be minimal generators of $I$, that is, $d_1$ and $d_2$ determine the size of the boxes associated to $I$. For $k\ge 1$ consider $(I^k)^l=I^{kl}$. Assume that $m$ is a minimal generator of $I^{kl}$. Since $I$ is a good ideal, we know that $m\in B_{a_1,a_2}$ for some $a_1$ and $a_2$ such that $a_1+a_2=kl-1$. From Proposition~\ref{prop} we also know that $m\in B^k_{b_1,b_2}$, where $b_i=\lfloor\frac{a_i}{k}\rfloor$. We want to show that $b_1+b_2=l-1$. On the one hand,
$$b_1+b_2=\left\lfloor\frac{a_1}{k}\right\rfloor+\left\lfloor\frac{a_2}{k}\right\rfloor\le\left\lfloor\frac{a_1+a_2}{k}\right\rfloor=\left\lfloor\frac{kl-1}{k}\right\rfloor=l-1,$$
on the other hand,
\begin{multline*}
b_1+b_2\ge\frac{a_1-k+1}{k}+\frac{a_2-k+1}{k}=
\frac{a_1+a_2-2k+2}{k}=\frac{kl-1-2k+2}{k}=\\=\frac{k(l-2)+1}{k}>l-2.
\end{multline*}
 Since $b_1+b_2$ is an integer, we get $b_1+b_2=l-1$ which finishes the proof.
\end{proof}
\begin{remark}
Another way to prove this proposition is to note that for $n=2$ the necessary and sufficient conditions for being a good ideal coincide: $I$ is a good ideal if and only if for any minimal generator $x^{\alpha_1}y^{\alpha_2}$ the following holds: $$\frac{\alpha_1}{d_1}+\frac{\alpha_2}{d_2}\ge1.$$ Since $I$ is a good ideal, for each minimal generator of $I$ the condition above holds. Minimal generators of $I^k$ are products of $k$ minimal generators of $I$, that is, if $x^{\beta_1}y^{\beta_2}$ is a minimal generator of $I^k$, we have $\frac{\beta_1}{d_1}+\frac{\beta_2}{d_2}\ge k$, that is, $\frac{\beta_1}{kd_1}+\frac{\beta_2}{kd_2}\ge 1$,  which means that $I^k$ is a good ideal as well.
\end{remark}
Let $I$ be a good ideal, let $m$ be some monomial. Note that among boxes containing $m$ there is always the largest box in the sense of partial order on coordinates of boxes: there exists a box $B_{s_1,\ldots,s_n}$ containing $m$ such that, if $B_{t_1,\ldots,t_n}$ also contains $m$, then $s_i\ge t_i$ for all $i$.
\begin{proposition}
\label{prop2}
Let $I$, $k$, $a_i$, $b_i$ be as in Proposition~\ref{prop}. Let $m$ be a monomial. Then the following holds:
\begin{itemize}
\item if $B_{a_1,\ldots,a_n}$ is the unique box containing $m$ among boxes associated to $I$, then $B^k_{b_1,\ldots,b_n}$ is the unique box containing $m$ among boxes associated to $I^k$;
\item if $B_{a_1,\ldots,a_n}$ is the largest box containing $m$ among boxes associated to $I$, then $B^k_{b_1,\ldots,b_n}$ is the largest box containing $m$ among boxes associated to $I^k$.
\end{itemize}
\end{proposition}
\begin{proof}
\hfill
\begin{itemize}
\item $m$ is contained in a unique box, associated to $I$ if and only if none of the coordinates of $m$ is divisible by the corresponding $d_i$. Then none of the coordinates of $m$ is divisible by the corresponding $kd_i$ or, equivalently, $m$ is contained in a unique box, associated to $I^k$. The rest follows from Proposition~\ref{prop}.
\item Let $m=x_1^{\alpha_1}\cdots x_n^{\alpha_n}$. We know that $B_{a_1,\ldots,a_n}\subseteq B^k_{b_1,\ldots,b_n}$, but we assume that $B^k_{b_1,\ldots,b_n}$ is not the largest one, that is, there is some $B^k_{c_1,\ldots,c_n}$ containing $m$ with at least one $i$ such that $c_i>b_i$. Without loss of generality we assume that $i=1$. Then $\alpha_1\ge c_1kd_1\ge (b_1+1)kd_1=(b_1k+k)d_1\ge (a_1+1)d_1$, that is, $B_{a_1,\ldots,a_n}$ is not the largest box containing $m$ among boxes associated to $I$, which is a contradiction.
\end{itemize}
\end{proof}
Henceforth, by old boxes we will mean boxes associated to $I$ and by new boxes we will mean boxes associated to $I^k$. Clearly, a new box consists of $k^n$ old boxes.
\begin{proposition}
\label{bad4}
Let $I\in \mathbb K[x_1,\ldots,x_n]$ be a good ideal and $n\ge 4$. Then for all $k\ge 2$, $I^k$ is a bad ideal.
\end{proposition}
\begin{proof}
Consider $I^k$, $k\ge 2$. We will show that the box decomposition principle fails already in $(I^k)^2=I^{2k}$. Let $x_1^{d_1}=\mu_1,\ldots,x_n^{d_n}=\mu_n$ be minimal generators of $I$. Then $m=\mu_1\mu_2\mu_3^{k-1}\mu_4^{k-1}\in I^{2k}$. The largest old box containing $m$ is $B_{1,1,k-1,k-1,0,\ldots,0}$, therefore, by Proposition~\ref{prop2}, the largest new box containing $m$ is $B^k_{\lfloor\frac{1}{k}\rfloor,\lfloor\frac{1}{k}\rfloor,\lfloor\frac{k-1}{k}\rfloor,\lfloor\frac{k-1}{k}\rfloor,0,0,\ldots,0}=B^k_{0,0,\ldots,0}$ and $0+0+\ldots+0<1$. Therefore, by Proposition~\ref{remeq}, $I$ is a bad ideal.
\end{proof}

The case $n=3$ is special in the following sense:
\begin{proposition}
\label{bad3}
Let $I\in \mathbb K[x,y,z]$ be a good ideal. Then the following holds:
\begin{itemize}
\item if $I$ has exactly three generators, that is, $I=\langle \mu_1, \mu_2, \mu_3 \rangle$, then $I^2$ is a good ideal and all $I^k$, $k\ge 3$, are bad;
\item if $I$ has more than three generators, then all $I^k$, $k\ge 2$, are bad.
\end{itemize}
\end{proposition}
\begin{proof}
\hfill
\begin{itemize}
\item Assume that $I=\langle \mu_1, \mu_2, \mu_3\rangle$. Consider $(I^2)^l=I^{2l}$. According to Proposition~\ref{remeq}, it is enough to show that every $m=\mu_1^{l_1}\mu_2^{l_2}\mu_3^{l_3}$, where $l_1+l_2+l_3=2l$, can be placed into a new box with the sum of coordinates greater or equal to $l-1$. Since $m\in B_{l_1,l_2,l_3}$, we conclude that $m\in B^2_{\lfloor\frac{l_1}{2}\rfloor,\lfloor\frac{l_2}{2}\rfloor,\lfloor\frac{l_3}{2}\rfloor}$. Note that at least one of $l_i$ is even, say, $l_1$ is even. Then ${\lfloor\frac{l_1}{2}\rfloor+\lfloor\frac{l_2}{2}\rfloor+\lfloor\frac{l_3}{2}\rfloor}\ge \frac{l_1}{2}+\frac{l_2-1}{2}+\frac{l_3-1}{2}=l-1$. Therefore, $I^2$ is a good ideal. 

Now we want to show that $I^k$, $k\ge 3$, is a bad ideal. We will show that the box decomposition principle fails already in $(I^k)^2=I^{2k}$. Consider $m=\mu_1^2\mu_2^{k-1}\mu_3^{k-1}\in I^{2k}$. Then the largest old box containing $m$ is $B_{2,k-1,k-1}$. Therefore, by Proposition~\ref{prop2}, the largest new box containing $m$ is $B^k_{\lfloor\frac{2}{k}\rfloor,\lfloor\frac{k-1}{k}\rfloor,\lfloor\frac{k-1}{k}\rfloor}=B^k_{0,0,0}$ and $0+0+0<1$. By Proposition~\ref{remeq}, $I^k$ is a bad ideal.
\item Assume that $I$ has at least 4 generators and let $m$ be any minimal generator of $I$ which is not a corner. Consider $I^k$, $k\ge 2$. We will show that the box decomposition principle fails already in $(I^k)^2=I^{2k}$. Consider $m_1=m\mu_1\mu_2^{k-1}\mu_3^{k-1}\in I^{2k}$. This monomial belongs to $B_{1,k-1,k-1}$ (note that this is the unique old box containing $m_1$). By Proposition~\ref{prop2}, the unique new box containing $m_1$ is $B^k_{\lfloor\frac{1}{k}\rfloor,\lfloor\frac{k-1}{k}\rfloor,\lfloor\frac{k-1}{k}\rfloor}=B^k_{0,0,0}$ and $0+0+0<1$. By Proposition~\ref{remeq}, $I^k$ is a bad ideal.
\end{itemize}
\end{proof}
\begin{remark}
\label{rem3}
Let $I=\langle \mu_1,\ldots, \mu_n\rangle$. Then, summarizing the propositions above, we obtain that $I^k$ is good if and only if $k=1$ or $n\le 2$ or $(n,k)=(3,2)$. We will use this later when we discuss the connection to Freiman ideals in Section~11.
\end{remark}

Another natural question to ask is whether all powers of good ideals are Ratliff--Rush (even though powers of good ideals are bad in most cases). We will state a sufficient condition that will help us to construct a family of examples of ideals whose all powers are Ratliff--Rush.

\begin{lemma}
\label{lem2}
Let $I$ be a good ideal and let $a_1,\ldots,a_n,a,t$ be nonnegative integers such that $a_1+\ldots+a_n=a$ and $t\ge 1$. Then 
$$I^{a+t}:\langle\mu_1^{a_1}\cdots \mu_n^{a_n}\rangle=\sum_{\substack{t_1,\ldots,t_n\ge0\\t_1+\ldots+t_n=t-1}}\mu_1^{t_1}\cdots \mu_n^{t_n}I_{t_1+a_1,\ldots,t_n+a_n}.$$
In particular, if $a=0$, the equality above is the box decomposition of $I^t$ and if $t=1$, this is just Proposition~\ref{prop1}.
\end{lemma}
\begin{proof}
$\supseteq$ Let $m\in \mu_1^{t_1}\cdots \mu_n^{t_n}I_{t_1+a_1,\ldots,t_n+a_n} $ for some $t_1,\ldots, t_n$. Then $m=m_1\mu_1^{t_1}\cdots \mu_n^{t_n}$, where $m_1\in I_{t_1+a_1,\ldots,t_n+a_n}$, that is, $m_1\mu_1^{t_1+a_1}\cdots \mu_n^{t_n+a_n}\in I^{a+t}$. Therefore, 
$$m\mu_1^{a_1}\cdots \mu_n^{a_n}=m_1\mu_1^{t_1+a_1}\cdots \mu_n^{t_n+a_n}\in I^{a+t},$$ which implies $m\in I^{a+t}:\langle\mu_1^{a_1}\cdots \mu_n^{a_n}\rangle$.

$\subseteq$ $I^{a+t}:\langle\mu_1^{a_1}\cdots \mu_n^{a_n}\rangle = \left\langle\frac{m}{\gcd(m,\mu_1^{a_1}\cdots \mu_n^{a_n})}\mid m\in G(I^{a+t})\right\rangle$. Fix $m\in G(I^{a+t})$. Then $m\in B_{s_1,\ldots,s_n}$, $s_1+\ldots+s_n=a+t-1$, that is, $m=x_1^{s_1d_1+\alpha_1}\cdots x_n^{s_nd_n+\alpha_n}$, $0\le \alpha_i\le d_i$. We distinguish two cases considering the differences $s_i-a_i$.
\begin{enumerate}
\item 
Assume that for all $i$ we have $s_i\ge a_i$. Put $t_i:=s_i-a_i$. Then $t_1+\ldots+t_n=t-1$ and 
\begin{multline*}
\frac{m}{\gcd(m,\mu_1^{a_1}\cdots \mu_n^{a_n})}=\frac{m}{\mu_1^{a_1}\cdots \mu_n^{a_n}}=x_1^{(s_1-a_1)d_1+\alpha_1}\cdots x_n^{(s_n-a_n)d_n+\alpha_n}=\\=x_1^{t_1d_1+\alpha_1}\cdots x_n^{t_nd_n+\alpha_n}=\mu_1^{t_1}\cdots \mu_n^{t_n}\cdot x_1^{\alpha_1}\cdots x_n^{\alpha_n}\in \\ \in \mu_1^{t_1}\cdots \mu_n^{t_n}I_{t_1+a_1,\ldots,t_n+a_n},
\end{multline*}

since $x_1^{\alpha_1}\cdots x_n^{\alpha_n}\in I^{a+t}:\langle \mu_1^{s_1}\cdots \mu_n^{s_n}\rangle=I_{s_1,\ldots,s_n}=I_{t_1+a_1,\ldots,t_n+a_n}$.

\item 
Assume that among differences $s_i-a_i$ we have $r$ nonnegative and $n-r$ negative ones, $n-r\ge 1$. Without loss of generality, 
$s_1-a_1\ge 0, \ldots, s_r-a_r \ge 0, s_{r+1}-a_{r+1}<0, \ldots, s_n-a_n<0$. Then 
$\frac{m}{\gcd(m,\mu_1^{a_1}\cdots \mu_n^{a_n})}
=x_1^{(s_1-a_1)d_1+\alpha_1}\cdots x_r^{(s_r-a_r)d_r+\alpha_r}
=\mu_1^{s_1-a_1}\cdots \mu_n^{s_r-a_r}\cdot x_1^{\alpha_1}\cdots x_r^{\alpha_r}$. We know that $(s_1-a_1)+\ldots +(s_n-a_n)=a+t-1-a=t-1$, but since at least one difference is negative, we obtain $(s_1-a_1)+\ldots+(s_r-a_r)\ge t$. At least one of these differences is positive, say, $s_1-a_1\ge 1$. Therefore, $\mu_1^{s_1-a_1}\cdots \mu_r^{s_r-a_r}\cdot x_1^{\alpha_1}\cdots x_r^{\alpha_r}$ is divisible by some $\mu_1^{t_1+1}\mu_2^{t_2}\cdots \mu_r^{t_r}$, where $t_1+\ldots+t_r=t-1$. We put $t_{r+1}=t_{r+2}=\ldots=t_n=0$. 
Then $\mu_1^{t_1+1}\mu_2^{t_2}\cdots \mu_r^{t_r}= \mu_1\cdot \mu_1^{t_1}\mu_2^{t_2}\cdots \mu_r^{t_r}\mu_{r+1}^{t_{r+1}}\cdots \mu_n^{t_n}\in \mu_1^{t_1}\cdots \mu_n^{t_n}I_{a_1+t_1,\ldots,a_n+t_n}$, where $t_1+\ldots+t_n=t-1$, since $\mu_1\in I_{a_1+t_1,\ldots,a_n+t_n}$. This finishes the proof.

\end{enumerate}
\end{proof}
Recall that given a good ideal $I$, there exists a coloring as in Theorem~\ref{th2}, given by a finite disjoint union of cones. Let $L$ be the number defined in the end of Section~5.
\begin{proposition}
\label{prop3}
Let $I$ be a good ideal. Then $I^k$ is Ratliff--Rush for all $k\ge L+1$.
\end{proposition}
\begin{proof}
Let $k\ge L+1$. We want to show that $I^{kl+k}:I^{kl}=I^k$ for all $l\ge 0$. It is clear that $I^{kl+k}:I^{kl}\supseteq I^k$.
For the other inclusion note that $I^{kl+k}:I^{kl}\subseteq I^{kl+k}:\langle \mu_1, \ldots, \mu_n\rangle^{kl}$. Thus it is sufficient to show that $I^{kl+k}:\langle \mu_1, \ldots, \mu_n\rangle^{kl}\subseteq I^k$. We know that $I^{kl+k}:\langle \mu_1, \ldots, \mu_n\rangle^{kl}$ equals the intersection of ideals of the type $I^{kl+k}:\langle\mu_1^{k_1}\cdots \mu_n^{k_n}\rangle$, where $k_1+\ldots+k_n=kl$. From Lemma~\ref{lem2} we know that 
$$I^{kl+k}:\langle\mu_1^{k_1}\cdots \mu_n^{k_n}\rangle=\sum_{\substack{t_1,\ldots,t_n\ge0\\t_1+\ldots+t_n=k-1}}\mu_1^{t_1}\cdots \mu_n^{t_n}I_{t_1+k_1,\ldots,t_n+k_n}.$$
Therefore, $$I^{kl+k}:\langle \mu_1, \ldots, \mu_n\rangle^{kl}= \bigcap_{\substack{k_1,\ldots,k_n\ge 0\\k_1+\ldots+k_n=kl}}\left(
\sum_{\substack{t_1,\ldots,t_n\ge0\\t_1+\ldots+t_n=k-1}}\mu_1^{t_1}\cdots \mu_n^{t_n}I_{t_1+k_1,\ldots,t_n+k_n}\right).$$

Let $t_1,\ldots, t_n, t'_1,\ldots, t'_n$ be nonnegative integers such that $t_1+\ldots+t_n=t'_1+\ldots+t'_n=k-1$.
Every minimal generator of $\mu_1^{t_1}\cdots \mu_n^{t_n}I_{t_1+k_1,\ldots,t_n+k_n}$ is of the form $\mu_1^{t_1}\cdots \mu_n^{t_n}m$, where $m\in G(I_{t_1+k_1,\ldots,t_n+k_n})$ and similarly every minimal generator of $\mu_1^{t'_1}\cdots \mu_n^{t'_n}I_{t'_1+k'_1,\ldots,t'_n+k'_n}$ is of the form  $\mu_1^{t'_1}\cdots \mu_n^{t'_n}m'$, where $m'\in G(I_{t'_1+k'_1,\ldots,t'_n+k'_n})$. 
Thus any minimal generator of $(\mu_1^{t_1}\cdots \mu_n^{t_n}I_{t_1+k_1,\ldots,t_n+k_n})\bigcap \\ \bigcap (\mu_1^{t'_1}\cdots \mu_n^{t'_n}I_{t'_1+k'_1,\ldots,t'_n+k'_n})$ is of the form $\lcm(\mu_1^{t_1}\cdots \mu_n^{t_n}m, \mu_1^{t'_1}\cdots \mu_n^{t'_n}m')$, which is divisible by 
$$\lcm(\mu_1^{t_1}\cdots \mu_n^{t_n}, \mu_1^{t'_1}\cdots \mu_n^{t'_n})=\mu_1^{\max(t_1,t'_1)}\cdots \mu_n^{\max(t_n,t'_n)}.$$ 
Note that if $(t_1,\ldots, t_n)\not=(t'_1,\ldots, t'_n)$, then $\max(t_1,t'_1)+\ldots+\max(t_n,t'_n)\ge k$ and so 
$$\mu_1^{t_1}\cdots \mu_n^{t_n}I_{t_1+k_1,\ldots,t_n+k_n}\bigcap \mu_1^{t'_1}\cdots \mu_n^{t'_n}I_{t'_1+k'_1,\ldots,t'_n+k'_n}\subseteq I^k.$$
Therefore, it remains to consider the sum of intersections which share a common \\$(t_1,\ldots,t_n)$, that is, we are left with 
$$\sum_{\substack{t_1,\ldots,t_n\ge0\\t_1+\ldots+t_n=k-1}}\left(\bigcap_{\substack{k_1,\ldots,k_n\ge 0\\k_1+\ldots+k_n=kl}}\mu_1^{t_1}\cdots \mu_n^{t_n}I_{t_1+k_1,\ldots,t_n+k_n}\right).$$
Fix $t_1,\ldots,t_n$. Then $(t_1,\ldots,t_n)$ belongs to some cone $C$ of our coloring. This cone can not have dimension 0. Indeed, if we assume the opposite, then the vertex of this cone is $(t_1,\ldots,t_n)$ itself, but $t_1+\ldots+t_n=k-1\ge L$ and, given the way $L$ was defined, the only case when this could potentially be possible is  $t_1+\ldots+t_n=L$. But, as mentioned in  Remark~\ref{rem2}, $L$ can not be attained at a zero dimensional cone, therefore, we get a contradiction. Thus the dimension of $C$ is at least 1. Then at least one coordinate of the vertex of $C$, say, the first one, is not underlined. Then $(t_1+kl,t_2,\ldots,t_n)\in C$ as well as $(t_1,\ldots,t_n)$ and thus $I_{t_1+kl,t_2,\ldots,t_n}=I_{t_1,\ldots,t_n}$. 
Then taking $k_1=kl, k_2=\ldots=k_n=0$ we obtain
\begin{multline*}
\bigcap_{\substack{k_1,\ldots,k_n\ge 0\\k_1+\ldots+k_n=kl}}\mu_1^{t_1}\cdots \mu_n^{t_n}I_{t_1+k_1,\ldots,t_n+k_n}\subseteq \\ \subseteq \mu_1^{t_1}\cdots \mu_n^{t_n}I_{t_1+kl,t_2,\ldots,t_n}=\mu_1^{t_1}\cdots \mu_n^{t_n}I_{t_1,\ldots,t_n}.
\end{multline*}

Summing over all $(t_1,\ldots,t_n)$, we obtain
\begin{multline*}
\sum_{\substack{t_1,\ldots,t_n\ge0\\t_1+\ldots+t_n=k-1}}\left(\bigcap_{\substack{k_1,\ldots,k_n\ge 0\\k_1+\ldots+k_n=kl}}\mu_1^{t_1}\cdots \mu_n^{t_n}I_{t_1+k_1,\ldots,t_n+k_n}\right)\subseteq \\ \subseteq \sum_{\substack{t_1,\ldots,t_n\ge0\\t_1+\ldots+t_n=k-1}}\mu_1^{t_1}\cdots \mu_n^{t_n}I_{t_1,\ldots,t_n}=I^k,
\end{multline*}
which finishes the proof.
\end{proof}
\begin{example}
Let $I$ be an ideal of $\mathbb K[x,y]$ generated by all the possible monomials of degree $2d$ except $x^dy^d$, $d\ge 2$. Then $I$ is not Ratliff--Rush, but all higher powers are. Indeed, $I$ is a good ideal by the sufficient condition and $I_{a_1,a_2}=I'$ for all $(a_1,a_2)\not = (0,0)$, where $I'=I+\langle x^dy^d \rangle$. Then we can choose the following coloring: $\{C_{1,0}, C_{\underline{0},1}, C_{\underline{0},\underline{0}}\}$. Therefore, $L=1$ and thus all powers of $I$ except $I$ itself are Ratliff--Rush. $I$ is not Ratliff--Rush since $\tilde I = I'$.
\end{example}

So now we know that if we have a good ideal and its coloring with the maximal sum of coordiates of vertices of cones equal to $L$, then all $I^{k}$ with $k\ge L+1$ are Ratliff--Rush. It is natural to mention ideals for which $L=0$, or, in other words, ideals for which there is a coloring consisting of a single cone $C_{0,0,\ldots,0}$.
\begin{definition}
Let $I$ be a good ideal such that for all $(a_1,\ldots,a_n)$ we have $I_{a_1,\ldots,a_n}=I$. Then we call $I$ a\textbf{ very good } ideal. 
\end{definition}
Clearly, any ideal in $\mathbb K[x]$ is a very good one. If $I$ is a very good ideal, then from Proposition~\ref{prop3} all its powers are Ratliff--Rush. It is also clear that if $I$ is a very good ideal, then in particular $I^2=IJ$, where $J=\langle \mu_1,\ldots,\mu_n\rangle$. We will now show that these conditions are in fact equivalent.
\begin{proposition}
\label{prop4}
Let $I$ be an $\mathfrak{m}$-primary monomial ideal in $\mathbb K[x_1,\ldots,x_n]$. Then $I$ is very good if and only if $I^2=IJ$, where $J=\langle \mu_1,\ldots,\mu_n\rangle$.
\end{proposition}
\begin{proof}
As mentioned before, one implication is trivial. For the other implication, assume that $I^2=IJ$, where $J=\langle \mu_1,\ldots,\mu_n\rangle$. Then for any $k\ge 1$ we have $$I^k=J^{k-1}I=\sum_{k_1+\ldots+k_n=k-1}\mu_1^{k_1}\cdots \mu_n^{k_n}I.$$
Thus the box decomposition principle holds and moreover $I_{a_1,\ldots,a_n}=I$ for all $(a_1,\ldots,a_n)$.
\end{proof}

\begin{example}
Let $I$ be an $\mathfrak m$-primary monomial ideal in $\mathbb K[x_1,\ldots,x_n]$, $n\ge 2$ such that the following holds: there is a pair of indices $i\not=j$ such that for each minimal generator of $I$, except the corners $\mu_1,\ldots, \mu_n$, its $x_i$-exponent is greater or equal to $d_i/2$ and its $x_j$-exponent is greater or equal to $d_j/2$. Then products of non-corners are not needed as minimal generators of $I^2$, thus $$I^2=\sum_{k_1+\ldots+k_n=1}\mu_1^{k_1}\cdots \mu_n^{k_n}I=IJ,$$ where, as before, $J=\langle\mu_1,\ldots,\mu_n\rangle$. Therefore, $I$ is a very good ideal and thus all its powers are Ratliff--Rush. This example generalizes the example
 constructed in \cite{AS}.
\end{example}
\begin{proposition}
\label{prop5}
Let $I=\langle \mu_1,\ldots, \mu_n\rangle$. Then $I^k$ is good if and only if it is very good.
\end{proposition}
\begin{proof}
One way to prove this statement is to consider all the cases for which $I^k$ is good (Remark~\ref{rem3}) and directly check that $I^{2k}=I^kJ$ in each case. Note that here $J=\langle \mu_1^k,\ldots, \mu_n^k\rangle$.

 Another way is the following. Since generators of $I^k$ are monomials in variables $\mu_1,\ldots,\mu_n$, we can, without loss of generality, assume $\mu_1=x_1,\ldots,\mu_n=x_n$, even if $I$ is not equigenerated. Let $I^k$ be a good ideal and let $B^k_{a_1,\ldots, a_n}$ denote the boxes, associated to $I^k$. It is enough to show that $I^k_{1,0,\ldots,0}=I^k$, all other $I^k_{a_1,\ldots,a_n}$ with $a_1+\ldots+a_n=1$ are analogous (by $I^k_{a_1,\ldots,a_n}$ we mean $(I^k)_{a_1,\ldots,a_n}$ and not $(I_{a_1,\ldots,a_n})^k$). $I^k$ is minimally generated by all monomials of degree $k$ and $I^{2k}$ is minimally generated by all monomials of degree $2k$. By definition, $I^k_{1,0,\ldots,0}$ is generated by monomials in $B^k_{1,0\ldots, 0}\cap G(I^{2k})$, divided by $\mu_1^k=x_1^k$. In other words, we take all monomials of degree $2k$, which are divisible by $x_1^k$, and divide them by $x_1^k$. Clearly, we obtain all monomials of degree $k$ (and only them). Therefore, $I^k_{1,0,\ldots,0}=I^k$ and thus $I^k$ is a very good ideal.
\end{proof}
We remark that this proposition means the following: for $I^k=\langle \mu_1,\ldots, \mu_n\rangle^k$ there is nothing "extra" in boxes $B^k_{1,0\ldots, 0}, \ldots, B^k_{0,\ldots,0,1}$ besides the translations of $I^k$ and the only reason $I^k$ might fail to be a very good ideal is that it fails to be a good ideal, that is, there exists a monomial in $G(I^{2k})$ which only belongs to $B^k_{0,0,\ldots,0}$. Therefore, ideals of such type are either good (and thus very good), or the box decomposition principle should fail already in $(I^k)^2=I^{2k}$. If we look back at the proofs of Proposition~\ref{good2}, Proposition~\ref{bad4} and Proposition~\ref{bad3}, we see that it is indeed the case.

\section{Connection to Freiman ideals}
Let $I$ be an equigenerated monomial ideal with analytic spread $l(I)$. It has been shown in \cite{NG} (Theorem 1.9) that $|G(I^2)|\ge l(I)|G(I)|-{l(I)\choose 2}$. Note that this bound is no longer valid if $I$ is not equigenerated. Indeed, for each $m\ge 6$ there exists a monomial ideal in two variables such that $|G(I)|=m$ and $|G(I^2)|=9$, see \cite{TS}.
\begin{definition}
An equigenerated monomial ideal is called \textbf{Freiman} if 
$$|G(I^2)|= l(I)|G(I)|-{l(I)\choose 2}.$$
\end{definition}
The next proposition is exactly Theorem~2.1 in \cite{FR}, but we give an alternative proof.
\begin{theorem}
\label{th3}
Let $I\subset \mathbb K[x_1,\ldots,x_n]$ be an $\mathfrak m$-primary equigenerated monomial ideal. Then $I$ is Freiman if and only if $I$ satisfies the equivalent conditions from Proposition~\ref{prop4}. In other words, for an  $\mathfrak m$-primary equigenerated monomial ideal $I$ the following are equivalent:
\begin{enumerate}[(1)]
\item $I$ is Freiman;
\item $I$ is very good;
\item $I^2=IJ$,
\end{enumerate} 
where, as before, $J=\langle \mu_1,\ldots,\mu_n\rangle$ and $\mu_i=x_i^d$. 
\end{theorem}
\begin{proof}
Since $I$ is $\mathfrak m$-primary, we have $l(I)=n$. The ideal $I$ is equigenerated, say, in degree $d$, thus $I^2$ is equigenerated in degree $2d$. Note that any product of two elements of $G(I)$ gives us an element of $G(I^2)$. Indeed, if such a product is not in $G(I^2)$, then there is an element in $G(I^2)$ strictly dividing it. But this is impossible since these monomials are both of degree $2d$. Note that different products of two elements from $G(I)$ might of course give us the same monomial from $G(I^2)$. The idea is to list all the different elements from $G(I^2)$ and count them. Clearly, $\mu_1G(I):=\{\mu_1m \mid m\in G(I)\}$ is a set of $|G(I)|$ different elements of $G(I^2)$. We can similarly define sets $\mu_2G(I), \ldots, \mu_nG(I)$. Monomials inside each such set are different. Every pair of different sets, say $\mu_iG(I)$ and $\mu_jG(I)$, has a unique monomial in common, which is $\mu_i\mu_j$. Every triple of different sets has an empty intersection. Thus in all these sets we have $n|G(I)|-{n\choose 2}=l(I)|G(I)|-{l(I)\choose 2}$ different elements from $G(I^2)$. Now it is clear that the Freiman equality holds if and only if there are no other minimal generators of $I^2$ except those in $\mu_1G(I)\cup\ldots\cup\mu_nG(I)$ if and only if $I^2=\langle\mu_1\rangle I+\ldots+\langle\mu_n\rangle I=IJ$.

\end{proof}

\begin{remark}
Note that if $I$ is a very good, but not an equigenerated ideal, the Freiman equality $|G(I^2)|= l(I)|G(I)|-{l(I)\choose 2}$ still holds, but it is not of any particular interest and thus such ideals are not called Freiman. If $I$ satisfies the Freiman equality, but is not an equigenerated ideal, it is not necessarily very good and not even necessarily good. Consider, for instance, $I=\langle x^3,y^3,xy\rangle$.
\end{remark}
\begin{remark}
From Theorem~\ref{th3} we know that an equigenerated $\mathfrak m$-primary monomial ideal is Freiman if and only if it is very good. From Proposition~\ref{prop5} we also know that $I^k=\langle\mu_1,\ldots,\mu_n\rangle^k$, $\mu_i=x_i^{d_i}$, is very good if and only if it is good. Finally, from Remark~\ref{rem3} we know that $I^k=\langle\mu_1,\ldots,\mu_n\rangle^k$, $\mu_i=x_i^{d_i}$, is good if and only if $k=1$, $n\le2$, or $(n,k)=(3,2)$. Therefore, we conclude the following: $I^k=\langle \mu_1,\ldots,\mu_n\rangle^k$ is Freiman (note that here $\mu_i=x_i^d$ since we want the ideal to be equigenerated) if and only if $k=1$, $n\le2$, or $(n,k)=(3,2)$. This can be seen as an alternative proof of Theorem~2.3 in \cite{FR}.

\end{remark}

\bibliographystyle{siam}
\bibliography{abc}

\end{document}